\documentclass[12pt]{amsart}
\usepackage{amssymb}
\usepackage{amsmath,amscd}
\usepackage{amsfonts,latexsym,verbatim,amscd,mathrsfs,color,array}
\usepackage[all,cmtip]{xy}
\usepackage{mathrsfs}
\usepackage{multirow}
\usepackage{graphicx}
\graphicspath{ {./images/} }
\usepackage{amsfonts, amsthm}
\usepackage{bbm}
\usepackage[hmargin=1in,vmargin=1in]{geometry}
\usepackage{mathdots}
\usepackage{stmaryrd}
\usepackage{MnSymbol}
\usepackage{bbm}
\usepackage[hidelinks]{hyperref}
\usepackage{enumitem}
\usepackage[all]{xy}
\usepackage{tikz-cd}

\allowdisplaybreaks

\def\XXint#1#2#3{{\setbox0=\hbox{$#1{#2#3}{\int}$ }
		\vcenter{\hbox{$#2#3$ }}\kern-.6\wd0}}

\newcommand{\transint}{\cap\kern-0.63em|\kern0.7em}

\DeclareMathSymbol{\intprod}{\mathbin}{MnSyC}{'270}

\newcommand{\p}{{ \partial}}

\newcommand{\N}{{\mathbb N}}

\renewcommand{\P}{{\mathbb P}}

\newcommand{\R}{{\mathbb R}}

\renewcommand{\H}{{\mathbb H}}

\newcommand{\del}{{\partial}}

\renewcommand{\p}{{\partial}}
\newcommand{\eps}{{\varepsilon}}

\newtheorem{thm}{Theorem}[section]
\newtheorem*{thm*}{Theorem}
\newtheorem{lemma}[thm]{Lemma}
\newtheorem*{lemma*}{Lemma}
\newtheorem{prop}[thm]{Proposition}

\newtheorem{cor}[thm]{Corollary}
\newtheorem*{cor*}{Corollary}

\newtheorem*{conj*}{Conjecture}

\newtheorem{fact}[thm]{Fact}

\newenvironment{claim}{\par\medskip\noindent\textit{Claim.}\space}{\par\medskip}
\newenvironment{claimproof}{\par\noindent\textit{Proof of claim.}\space}{\hfill$\diamond$\medskip\par}

   \newtheoremstyle{others}
     {7pt}
     {6pt}
     {}
     {}
     {\bf}
     {.}
     {.5em}
     {}

\theoremstyle{others}
\newtheorem{rmk}[thm]{Remark}
\newtheorem{rmk*}[thm]{Remark}
\newtheorem{defn}[thm]{Definition}
\newtheorem{example}[thm]{Example}

\newtheorem*{question*}{Question}

\numberwithin{equation}{section}

\title[Finiteness of BMS Measure]{Finiteness of Bowen--Margulis--Sullivan Measure for Gromov--Patterson--Sullivan Systems}
\author{Rou Wen}
\address{Department of Mathematics, University of Wisconsin, Madison}
\email{rwen5@wisc.edu}

\begin{document}
\begin{abstract}
    In this paper, we develop a notion of \emph{strongly positive reccurent} (SPR) property for a convergence group with a continuous Gromov--Patterson--Sullivan (GPS) system defined by Blayac--Canary--Zhang--Zimmer. We prove that these SPR groups admits a finite Bowen--Margulis--Sullivan (BMS) measure on some associated flow spaces, which means that dynamically they admit a cocompact action on the flow spaces. This notion of SPR groups gives rise to many new examples of subgroups in higher rank Lie group that admit finite BMS measure beyond relatively Anosov groups.
\end{abstract}
\maketitle

\section{Introduction}

The Bowen–Margulis–Sullivan (BMS) measure plays a central role in the study of dynamical systems arising from groups acting on boundaries and the associated flow spaces. Margulis in his thesis \cite{alma991086072856206532} showed that when it is finite, the BMS measure is the unique measure that realizes maximal entropy for an Anosov flow space. Moreover, having a finite BMS measure often leads to mixing properties, counting results, and equidistribution of periodic orbits in the associated flow spaces in many settings (see e.g. Ricks in \cite{2014arXiv1410.3921R} for rank one CAT(0) spaces, Roblin \cite{MSMF_2003_2_95__1_0} for negatively curved manifolds, Blayac--Zhu \cite{2021arXiv210608079B} for rank one convex projective geometry, and Blayac--Canary--Zhu--Zimmer \cite{2024arXiv240409718B} and Kim--Oh \cite{KimOh+2025+91+142} for relative Anosov groups etc.). Thus it is important to characterize when this measure is finite.

For a geometrically finite Kleinian group $\Gamma \subset SO(d,1)$, Sullivan \cite{10.1007/BF02392379} showed that the BMS measure on the unit tangent bundle of the hyperbolic space $M:= \Gamma \backslash \H^d$ is finite. Later, Dal'bo--Otal--Peign\'e \cite{dalbootalpeigne} generalized this result to geometrically finite subgroups $(\Gamma, \mathcal{P})$ of the isometry group of negatively curved manifolds, provided that the critical exponent of each parabolic subgroup $P \in \mathcal{P
}$ is strictly smaller than the critical exponent of $\Gamma$. 

These results also have a higher rank analogy. Let $G$ be a semi-simple Lie group of non-compact type that has rank at least two. Kim--Oh \cite{KimOh+2025+91+142} showed finiteness for a whole family of BMS measures (parametrized by certain linear functionals on the Cartan subalgebra) associated to a relative Anosov subgroup of $G$. Wen \cite{2025arXiv250412448W} obtained similar result for relative Anosov group for a different flow space compared to the Kim--Oh construction. We will discuss the nuance in Section \ref{rel anosob}.

It seems that geometrically finiteness is a good candidate to characterize the finiteness of BMS measure. Sullivan proposed this question for hyperbolic spaces: does there exist a geometrically infinite subgroup of $SO(d,1)$ that admits a finite BMS measure? Peign\'e gave a positive answer to this question for $d>3$ in \cite{peigne}. This indicates that geometrically finiteness is a sufficient but not necessary condition for finite BMS measure. In fact, Pit--Schapira \cite{2016arXiv161003255S} developed a sufficient and necessary condition, namely positive recurrence, for finite BMS measure for general discrete subgroups of isometry group of pinched negatively curved Riemannian manifolds. 

 A natural question to ask then is: can a similar criterion for finiteness of BMS measure be established in the higher rank setting? One of the goals of this paper is to answer this question positively. In other words, we generalize the positive recurrence condition defined by Pit--Schapira into the higher rank setting, and show that it implies finite BMS measure.

 We actually work with the more general setting of Gromov–Patterson–Sullivan (GPS) systems for convergence groups introduced by Blayac--Canary--Zhu--Zimmer \cite{2024arXiv240409713B}. 
The precise definition of a GPS system is hard to summarize in this introduction, so we state our main results in terms of transverse subgroups of higher rank Lie group, which is one of the main applications of the GPS system framework. 

Let $G$ again be a semi-simple Lie group of non-compact type with rank $\geq 2$, and let $P_{\theta}$ be a parabolic subgroup of $G$ associated with a set of simple roots $\theta$. Fix a $KA^+K$ decomposition of $G$, where $K$ is a maximal compact subgroup in $G$, and $A^+$ is the image of the positive Weyl chamber $\mathfrak{a}^+$ of a Cartan subalgebra under the exponential map. Let $\kappa: G \rightarrow \mathfrak{a}^+$ denote the Cartan projection. A discrete subgroup $\Gamma \subset G$ is a transverse subgroup if $\alpha(\kappa(g_n)) $ diverges for all escaping sequences $\{g_n \} \subset \Gamma$ and all $\alpha \in \theta$, and its limit set $\Lambda_{\theta}(\Gamma)$ is a transverse subset of the partial flag manifold $\mathcal{F}_{\theta}:= G /P_{\theta}.$ 

Fix a linear functional
 $\phi \in (\mathfrak{a}^+)^*$ that is positive on the Benoist limit cone of $\Gamma$, then the associated critical exponent $\delta_{\phi}(\Gamma)<+\infty$, and there exists a Patterson--Sullivan density $\mu$ of dimension $\delta_{\phi}(\Gamma)$ supported on $\Lambda_{\theta}(\Gamma)$ \cite[Proposition 2.7 $\And$ 3.2]{2023arXiv230411515C}. Furthermore, $\phi$ defines a $\Gamma$ action on the flow space $$\tilde{\Omega}_{\Gamma}:= \Lambda_{\theta}(\Gamma)^{(2)} \times \R,$$ where $\Lambda_{\theta}(\Gamma)^{(2)}$ consists of transverse pairs in the product $\Lambda_{\theta}(\Gamma)\times \Lambda_{\theta}(\Gamma)$, and $\mu$ induces a BMS measure $m^{BMS}$ on the quotient $\Gamma \backslash\tilde{\Omega}_{\Gamma}$.

 The subgroup $\Gamma$ is said to be \emph{positive recurrent} with respect to some compact set $K \subset \tilde{\Omega}_{\Gamma}$ if the Poincar\'e series associated to $\Gamma$ and $\phi$ is divergent at $\delta_{\phi}(\Gamma)$, and $$\sum_{g \in \Gamma_K} \phi(\kappa(g)) \cdot \exp\left(-\delta_{\phi}(\Gamma)\cdot \phi(\kappa(g))\right) < +\infty.$$
 Here $\Gamma_K$ should be understood as a subset of $\Gamma$ that consists of elements ``outside" of $K$, the precise definition can be found in Section \ref{finite bms}.


\begin{cor*}[Theorem \ref{pr implies finite}]
    If there exists a compact set $K \subset \Tilde{\Omega}_{\Gamma}$ such that $\Gamma$ is positive recurrent with respect to $K$, then $m^{BMS}(\Gamma \backslash \tilde{\Omega}_{\Gamma})$ is finite.
\end{cor*}

We also introduce a stronger condition, called strongly positive recurrence (SPR). This condition is formulated in terms of the entropy at infinity $\delta_{\infty}(\Gamma)$ of the subgroup $\Gamma$, defined using the critical exponents of $\Gamma_K$ for large compact subsets of $\tilde{\Omega}_{\Gamma}$. The convergence group $\Gamma$ is said to be \emph{strongly positively recurrent} if $\delta_{\infty}(\Gamma) <\delta_{\phi}(\Gamma)$.

\begin{cor*}[Theorem \ref{spr implies finite BMS}]
    If $\Gamma$ is SPR, then $m^{BMS}(\Gamma \backslash \tilde{\Omega}_{\Gamma})$ is finite.
\end{cor*}

A key step in the proof of the above corollary is to show that the SPR condition forces the Poincaré series of $\Gamma$ to diverge at the critical exponent.

\begin{cor*}[Theorem \ref{spr implies divergent}]
    If $\Gamma$ is SPR, then the Poincar\'e series associates with $\Gamma$ and $\phi$ is divergent at $\delta_{\phi}(\Gamma).$
\end{cor*}


The SPR subgroups include relative Anosov groups, which are known to have finite BMS measures (\cite{KimOh+2025+91+142}, \cite{2025arXiv250412448W}, \cite{2024arXiv240409713B}).
We also showed that the entropy at infinity behaves nicely under Schottky products of convergence groups, which provides a way to construct new examples that are not relative Anosov whose associated flow spaces admit finite BMS measures. 

The paper is organized as follows. In Section 2 we review convergence group actions, expanding cocycles, Patterson–Sullivan measures, and GPS systems. In Section 3 we introduce the notion of positive recurrence and prove that it implies finiteness of the BMS measure. In Section 4 we define entropy at infinity and strongly positive recurrence, and we prove that SPR groups satisfy the positive recurrence condition and therefore admit finite BMS measure. In Section 5 we show relative Anosov groups are SPR, as well as construct new examples of SPR groups. 

\subsection*{Acknowledgment} 
I would like to thank Pierre-Louis Blayac for raising this question to me in the first place and having several helpful discussions later. I am very grateful to Andrew Zimmer for generously sharing his knowledge of the GPS system with me, and Barbara Schapira for her warm encouragements and explanation of her work. I also thank Alex Nolte and Fanny Kassel for giving me the inspiration for constructing Example \ref{projective geom}.

The author is supported by NSF grant DMS-2105580, Institut Henri Poincaré (UAR 839 CNRS-Sorbonne Université) and LabEx CARMIN (ANR-10-LABX-59-01).

\section{Background}

\subsection{Convergence Group Action}\label{conv group action}
 Let $M$ be a compact metrizable space. 
\begin{defn}[Convergence Group]
   A subgroup $\Gamma \subset Homeo(M)$ acts on $M$ as a \emph{convergence group} if for every sequence $\{\gamma_n\} $ consisting of distinct elements, there exist $x,y \in M$ and a subsequence $\{\gamma_{n_j}\}$ such that $\gamma_{n_j} |_{M\setminus \{y\}}$ converges locally uniformly to $x$.
\end{defn}

Let $\Gamma$ be a convergence group acting on $M$. Associated to $\Gamma, $ we can define the \emph{limit set}, denoted by $\Lambda_{\Gamma}\subset M$, as the set of points $x \in M$ where there exist $y \in M$ and a sequence $\{\gamma_n\} \subset \Gamma$ such that $\gamma_n|_{M\setminus \{y\}}$ converges locally uniformly to $x$. We call the group $\Gamma$ \emph{non-elementary} if its limit set $\Lambda_{\Gamma}$ consists of more than two points, we will assume $\Gamma$ is non-elementary from now on.

A point $x \in \Lambda_{\Gamma}$ is a \emph{conical limit point } if there exist distinct points $a,b \in M$ and a sequence $\{\gamma_n\}\subset \Gamma$ such that $\gamma_n.x$ converges to $a$ and $\gamma_n|_{M\setminus\{x\}}$ converges to $b$. We will use $\Lambda_{\Gamma}^{con}$ to denote the set of conical limit points in $\Lambda_{\Gamma}$. Furthermore, elements in $\Gamma$ can be classified by the type of limit points they have.

\begin{fact}[Classification of Elements \cite{Tukia}]
An element $\gamma \in \Gamma$ is either:
    \begin{itemize}
        \item \emph{loxodromic} if it has two distinct fixed points $\gamma^{\pm} \in \Lambda_{\Gamma}$ such that $\gamma^{\pm n}|_{M\setminus\{\gamma^{\mp}\}}$ converges locally uniformly to $\gamma^{\pm}$, i.e. $\gamma^{\pm}$ are conical limit points;
        \item \emph{parabolic} if it has one fixed point $p \in 
        \Lambda_{\Gamma}$ such that $\gamma^{\pm n}|_{M\setminus\{p\}}$ converges to $p$ locally uniformly, i.e. $p$ is not a conical limit point;
        \item \emph{elliptic} if it has finite order.
\end{itemize}
\end{fact}

\subsection{Continuous Cocycles}\label{cont cocycle}
We now define a continuous cocycle associated to $\Gamma$ and $M$. 
\begin{defn}[Continuous Cocycles]
    A function $\sigma: \Gamma \times M \rightarrow \R$ is called a \emph{continuous cocycle} if:
    \begin{enumerate}
        \item For every $\gamma \in \Gamma$, the function $\sigma(\gamma, \cdot)$ is continuous, i.e. for all $x_0 \in M$ $$\limsup_{x \rightarrow x_0} \sigma(\gamma, x)-\sigma(\gamma, x_0) =0;$$
        \item $\sigma$ satisfies the cocycle identity, i.e. for all $g, h \in \Gamma$ and all $x \in M$ $$\sigma(gh, x)-\sigma(g, h.x)-\sigma(h, x)=0$$
    \end{enumerate}   
\end{defn}

Due to a result of Blayac--Canary--Zhang--Zimmer \cite[Proposition 2.3]{2024arXiv240409713B}, we can put a metric $d $ on $\Gamma \sqcup M$ such that $\Gamma$ acts on it also as a convergence group. This metric induces a compactifying topology on $\Gamma \sqcup M$, in particular, $\Gamma$ is open in $\Gamma \sqcup M$ and its closure is $\Gamma \sqcup \Lambda_{\Gamma}$.

\begin{defn} With the compactifying metric $d$ on $\Gamma \sqcup M$, we can define that a continous cocycle $\sigma$ is :
\begin{itemize}
    \item \emph{proper} if for every distinct sequence of loxodromic elements $\{\gamma_n\}$ such that $$\liminf_{n\rightarrow \infty} d(\gamma_n^-, \gamma_n^+) >0,$$ we have $\sigma (\gamma_n, \gamma_n^+) \rightarrow +\infty$.
    \item \emph{expanding} if $\sigma$ is proper, and there exists a $\sigma-$magnitude $||\cdot||_{\sigma}:\Gamma \rightarrow \R$ with the property that for every $\eps >0$, there exists a $C>0$ such that $$||\gamma||_{\sigma}-C \leq \sigma(\gamma, x) \leq ||\gamma||_{\sigma}+C$$ for all $x \in M $ with $d(x, \gamma^{-1}) > \eps.$
\end{itemize}
\end{defn}

There are a few properties of a expanding cocycle that will be helpful later.
\begin{fact}\cite[Proposition 3.3]{2024arXiv240409713B}\label{prop of cocycle}
    Given $\Gamma \subset Homeo(M)$ a convergence group, and $\sigma$ an expanding continuous cocycle.
    \begin{itemize}
        \item For any finite set $F \subset \Gamma,$ there exist $D >0$ such that for all $f \in F$, $\gamma \in \Gamma$, we have $$||\gamma||_{\sigma}-D \leq ||f\gamma||_{\sigma}\leq ||\gamma||_{\sigma}+D$$ and $$||\gamma||_{\sigma}-D \leq ||\gamma f||_{\sigma}\leq ||\gamma||_{\sigma}+D.$$

        \item For any $\eps > 0$ there exists $C > 0$ such that: if $\alpha, \beta \in \Gamma$ and $d(\alpha^{-1},\beta) > \eps$, then 
        $$||\alpha||_{\sigma}+||\beta||_{\sigma}-C \leq ||\alpha\beta||_{\sigma}\leq ||\alpha||_{\sigma}+||\beta||_{\sigma}+C. $$
        \item For any escaping sequence $\{g_n \} \subset \Gamma$, we have $$\lim_{n \rightarrow \infty} ||g_n||_{\sigma} = +\infty.$$
        
        \item If $\{g_n\}$ is a escaping sequence, $\{x_n\} \subset M, $ and $\sigma(g_n, x_n)$ is bounded below, then $$\lim_{n \rightarrow \infty} d(g_n.x_n, g_n)=0.$$
    \end{itemize}
\end{fact}

\subsection{GPS System}
We can now define the Gromov--Patterson--Sullivan system associated to the continuous cocycle $\sigma$ and construct a flow space that $\Gamma$ acts on.
\begin{defn}
    Let $M^{(2)}:= \{(x,y)\in M\times M\text{ }|\text{ }x\neq y\}$. For $\sigma$ and $\bar{\sigma}$ two proper cocycle, and $G:M^{(2)} \rightarrow \R$ a locally bounded function, we say $(\sigma, \bar{\sigma}, G)$ is a continuous Gromov--Patterson--Sullivan (GPS) system if they satisfy $$\left[\bar{\sigma}(\gamma, y)+\sigma(\gamma, x)\right]-\left[G(\gamma x, \gamma y)-G(x,y)\right] =0$$ for all $\gamma \in \Gamma$, and distinct $x, y \in M$. 
\end{defn}

Let $(\sigma, \bar{\sigma}, G)$ be a continuous GPS system through the rest of the paper. 

\begin{fact}\cite[Proposition 3.4]{2024arXiv240409713B}\label{3.4} Here are some useful properties of $(\sigma, \bar{\sigma},G)$:
    \begin{itemize}
        \item $\sigma$ and $\bar{\sigma}$ are expanding cocycles.
        \item There exists $E\geq 0$ such that for all $\gamma \in \Gamma$ we have \begin{equation}\label{1}
            ||\gamma^{-1}||_{\bar{\sigma}}-E \leq ||\gamma||_{\sigma}\leq ||\gamma^{-1}||_{\bar{\sigma}}+E.
        \end{equation}
    \end{itemize}
\end{fact}

\subsection{A Flow Space}\label{flow space}

 One can construct a natural flow space that admits a measurable $\Gamma$ action. Let $\Tilde{\Omega}_{\Gamma}:= \Lambda_{\Gamma}^{(2)} \times \R$, where $\Lambda_{\Gamma}^{(2)}$ is the set of pairs of distinct points in $\Lambda_{\Gamma} \times \Lambda_{\Gamma}$. Fixing another expanding cocycle $\sigma'$. There is a natural flow $\Tilde{\psi}^s$ defined on $\Tilde{\Omega}_{\Gamma}$ by translating on the last factor, i.e. $$\Tilde{\psi}^s(x,y, t)= (x, y, t+s),$$ and the $\Gamma$ action on $\Tilde{\Omega}_{\Gamma}$ is defined by $$\gamma.(x,y,t)= \left(\gamma x, \gamma y, t+\sigma'(\gamma, x)\right).$$ Note that the $\Gamma$ action commutes with the flow $\Tilde{\psi}^s$ and hence $\Tilde{\psi}^s$ descends to a flow $\psi^s$ on $\Omega_{\sigma'}:= \Gamma \backslash\tilde{\Omega}_{\Gamma}$.

\begin{fact}\cite[Proposition 7.1]{2024arXiv240409718B}
    The space $M^{(2)} \times \R \supset \tilde{\Omega}_{\Gamma}$ admits a compactification $\overline{M^{(2)} \times \R} = (M^{(2)} \times \R) \sqcup M$. Moreover, a sequence $\{(x_n,y_n,t_n)\} \in \overline{M^{(2)} \times \R} $ converges to a point $z \in M$ if every subsequence contains a further subsequence, which is still denoted by $\{(x_n,y_n,t_n)\} $, that satisfies either \begin{itemize}
        \item $x_n \rightarrow z$ and $t_n \rightarrow \infty$,
        \item $y_n \rightarrow z$ and $t_n \rightarrow -\infty$, or
        \item $x_n \rightarrow z$ and $y_n \rightarrow z$.
    \end{itemize}
\end{fact}

\begin{rmk}
    With this topology, the points $\tilde{\psi}^s(v)$ along a flow line defined by $v=(v^+,v^-,t_v)$ converges to $v^{\pm}$ if $s \rightarrow \pm \infty.$ We will use $\tilde{\psi}^{\pm \infty}(v)$ to denote its limit points for the rest of the paper. 
\end{rmk}

In the classical hyperbolic space setting (or negatively curved spaces in general), this flow space can be understood as the unit tangent bundle of the space, and the one dimensional flow as the geodesic flow. The hyperbolicity of the underlying space in those cases gives rise to nice properties such as geodesics with close endpoints are close to each other in the middle. Even though the flow space $\tilde{\Omega}_{\Gamma}$ defined above admits a compactifying topology (and hence metrizable), it is not Gromov hyperbolic in general. Nevertheless we can still show it has similar properties .

Let $K \subset int(K')$ be two compact sets in $\tilde{\Omega}_{\Gamma}$ such that for all $g \in \Gamma$, there exist $v \in K$ such that the flow line defined by $v $ intersects $g.K$ (we will give an explicit construction of such sets in Section \ref{Gamma_K}. For any compact set $W \subset \tilde{\Omega}_{\Gamma}$, define $$t_{g, W}:= \inf \{t \in [0, \infty) \text{ }|\text{ } \exists v \in W \text{ with } \Tilde{\psi}^t(v) \in g.W \text{ and } \Tilde{\psi}^{[0,t]} (v)\cap \Gamma.W \subset W \cup g.W\}.$$

\begin{lemma}\label{disjoint flow lines}
 There exist $L>0$ and $R>0$ such that if $t_{g, K'} \geq L$, and $v \in K $ such that $\Tilde{\psi}^{t_v}(v) \in g.K$ for some $t_v >0$, we have $$\Tilde{\psi}^{[R, t_v-R]} (v)\cap \Gamma.K = \emptyset.$$ 
\end{lemma}

\begin{proof}
We prove the lemma by contradiction. Assume the statement fails, then there exist $v_n=(v_n^+,v_n^-,t_n) \in K, $ $g_n, \gamma_n \in \Gamma,$ and $t_n>t_{g_n,K'}, s_n >0$ such that $t_{g_n, K'} \geq n$, $\Tilde{\psi}^{t_n}(v_n) \in g_n.K$, $s_n \in [\frac{n}{2}, t_{v_n}-\frac{n}{2}]$, and $\Tilde{\psi}^{s_n}(v_n) \in \gamma_n.K$

By definition of $t_{g_n,K'}$, we can also find $w_n \in K'$ such that $$\Tilde{\psi}^{[0, t_{g_n, K'}]} (w_n)\cap \Gamma.K' \subset K' \cup g_n.K'.$$ 
Re-centering everything by $\gamma_n^{-1}$ we have $$\gamma_n^{-1}\left(\Tilde{\psi}^{s_n}(v_n)\right) \rightarrow v =(v^+,v^-,s_v) \in K,$$ $$\gamma_n^{-1}(v_n) \in \gamma_n^{-1}.K \subset \gamma_n^{-1}.K' ,$$ $$\gamma_n^{-1}\left(\Tilde{\psi}^{t_{v_n}}(v_n)\right) \in \gamma_n^{-1}g_n.K \subset \gamma_n^{-1}g_n.K'  .$$

By construction, $\gamma_n^{-1}.v_n^+ \rightarrow v^+$ and $\gamma_n^{-1}.v_n^- \rightarrow v^-$ and $\sigma'(\gamma_n^{-1},v_n^+)$ is bounded away from $-\frac{n}{2}.$ Then the above fact implies that $\gamma_n^{-1}(v_n)\rightarrow v^-$ and $\gamma_n^{-1}\left(\Tilde{\psi}^{t_{v_n}}(v_n)\right)  \rightarrow v^+$.  In particular, this implies $\gamma_n^{-1}w_n \in \gamma_n^{-1}.K' \rightarrow v^-$ and $\gamma_n^{-1}\Tilde{\psi}^{t_{g_n, K'}}(w_n) \in \gamma_n^{-1}g_n.K' \rightarrow v^+$ as $\Gamma$ also acts on $\overline{M^{(2)}\times \R}$ as a convergence group \cite[Proposition 7.2]{2024arXiv240409718B}. By construction $g_n \in \Gamma_{K'}$, this implies the flow line $(v^+,v^-,\R)$ does not intersect the interior of $K'$. However, there is a contradiction as $v \in (v^+,v^-,\R) \cap K.$
\end{proof}

 \begin{cor}\label{flow lines with same endpoint are close}
     There exists $L'$ such that if there exist $v \in K'$, $g \in \Gamma$, and $t_g >L$ satisfying $\tilde{\psi}^{t_g}(v) \in g.K$, then for all $v' \in K'$ with $\tilde{\psi}^{\infty}(v') = \tilde{\psi}^{\infty}(v)$, we have $\tilde{\psi}^s(v') \cap g.K' \neq \emptyset$ for some $s >0$.
 \end{cor}

\begin{proof}
    The proof idea is similar to that of Lemma \ref{disjoint flow lines}. If the lemma fails, then there exist $v_n , v_n' \in K'$, $g_n \in \Gamma$, and $t_n \rightarrow \infty$ such that $$\tilde{\psi}^{t_n}(v_n) \in g_n.K \text{, } \tilde{\psi}^{[0,\infty)}(v_n') \cap g_n.K' = \emptyset \text{, and } \tilde{\psi}^{\infty}(v_n)= \tilde{\psi}^{\infty}(v_n') =: x_n.$$
Pulling everything back by $g_n^{-1}$, we have $$g_n^{-1}\tilde{\psi}^{t_n}(v_n) \rightarrow v \in K, \text{ } g_n^{-1} x_n \rightarrow x^+ \in \Lambda_{\Gamma}.$$ Moreover, since $t_n \rightarrow \infty$, we have $g_n^{-1}.K' \rightarrow x^- \in \Lambda_{\Gamma}$ as argued before, hence $g_n^{-1}v_n$ and $g_n^{-1}(v_n')$ both converge to $x^-.$ The flow lines $g_n^{-1}\tilde{\psi}^{[0,\infty)}(v_n)$ converge to $\tilde{\psi}^{[0,\infty)}(v)$, implying $\tilde{\psi}^{\pm \infty}(v) = x^{\pm}$. This gives a contradiction. as $g_n^{-1}\tilde{\psi}^{[0,\infty)}(v_n')$ also converges to $\tilde{\psi}^{[0,\infty)}(v)$ but it cannot intersect $int(K').$
\end{proof}

\subsection{Patterson--Sullivan Measure and Bowen--Margulis--Sullivan Measure}\label{ps and bms measure}

One can construct the $\sigma-$Poincar\'e series $$Q_{\Gamma,\sigma}(s):= \sum_{\gamma \in \Gamma} \exp\left(-s\cdot ||\gamma||_{\sigma}\right),$$ and define the $\sigma-$critical exponent $\delta_{\sigma} := \inf \{s > 0 \text{ }|\text{ } Q_{\Gamma,\sigma}(s) < +\infty\}$. Note that Equation \ref{1} implies $\delta_{\bar{\sigma}} = \delta_{\sigma}$ for a GPS system $(\sigma,\bar{\sigma},G)$. To abuse notation, we will use $ \delta_{\sigma}$ to denote the critical exponent associated to both $\sigma$ and $\bar{\sigma}$ for the rest of the paper.

A probability measure $\mu$ on $M$ is a $\sigma-$Patterson--Sullivan measure of dimension $s$ if for every $\gamma \in \Gamma,$ the measures $\gamma_*\mu$ and $\mu $ are absolutely continuous and their Radon derivative satisfies $$ \frac{d(\gamma_*\mu)}{d\mu}(x) =e^{-s\cdot \sigma(\gamma^{-1}, x)} $$ for all $x \in M.$

Theorem 4.1 in \cite{2024arXiv240409713B} shows that if $\sigma$ is an expanding cocycle and $\delta_{\sigma}<+\infty$, then there exist a $\sigma-$Patterson--Sullivan measure $\mu$ of dimension $\delta_{\sigma}$. Moreover, $\mu$ is supported on the limit set $\Lambda_{\Gamma}$.

 We can construct a measure $\tilde{m}_{\sigma}$ on the flow space $\tilde{\Omega}_{\Gamma}$ defined above: $$\Tilde{m}_{\sigma}(x,y,t)= \exp{(s \cdot G(x,y))} d\mu(x) \otimes d\bar{\mu}(y) \otimes dt$$ where $\mu$ and $\bar{\mu}$ are the Patterson--Sullivan measure of dimension $s$ associated to $\sigma$ and $\bar{\sigma}$ respectively. This measure $\Tilde{m}_{\sigma}$ is both flow invariant and $\Gamma$ invariant, and Proposition 10.2 in \cite{2024arXiv240409713B} implies that $\Gamma$ acts properly discontinuously on $\Tilde{\Omega}_{\Gamma}$, hence it descends to a flow invariant measure $m_{\sigma} $ on the quotient space $\Omega_{\sigma'}$. We refer to this as the Bowen--Margulis--Sullivan (BMS) measure on $\Omega_{\sigma'}.$

Blayac--Canary--Zhang--Zimmer proved in \cite{2024arXiv240409713B} a Hopf--Tsuji--Sullivan dichotomy for the GPS system $(\sigma, \bar{\sigma}, G).$

\begin{thm}\cite[Theorem 1.8 \& 11.2]{2024arXiv240409713B}\label{dichotomy}
    Suppose $(\sigma, \bar{\sigma}, G)$ is a continuous GPS system with $\delta <+\infty$. Let $\mu$ and $\bar{\mu}$ be Patterson--Sullivan measures of dimension $s$ defined as before, and let $\nu := \exp{(s\cdot G)}\mu \otimes \bar{\mu}$ be a locally finite $\Gamma$ invariant measure on $M^{(2)}$. Then the following dichotomy holds:
    \begin{enumerate}
        \item If $Q_{\Gamma,\sigma}(s)=\infty$, then:
        \begin{itemize}
            \item $s= \delta_{\sigma}$;
            \item $\mu(\Lambda_{\Gamma}^{con})= \bar{\mu}(\Lambda_{\Gamma}^{con}) = 1$;
            \item The $\Gamma$ action on $(M^{(2)}, \nu)$ is ergodic and conservative. 
        \end{itemize}
        
        \item If $Q_{\Gamma,\sigma}(s)<+\infty$, then:
        \begin{itemize}
            \item $s \geq \delta_{\sigma}$;
            \item $\mu(\Lambda_{\Gamma}^{con})= \bar{\mu}(\Lambda_{\Gamma}^{con}) = 0$;
            \item The $\Gamma$ action on $(M^{(2)}, \nu)$ is non-ergodic and dissipative. 
        \end{itemize}
    \end{enumerate}

    Moreover, in the first case where the Poincar\'e series $Q_{\Gamma,\sigma}$ diverges at $\delta_{\sigma}$, the measures $\mu$ and $\bar{\mu}$, and consequently $\Tilde{m}_{\sigma}$, are unique, and the flow $\psi^s$ on $(\Omega_{\sigma'}, m_{\sigma})$ is ergodic and conservative.
\end{thm}

\subsection{Shadows}
Another important class of objects that we need is shadows.
\begin{defn}\label{shadow}
    Given $\eps >0$, and $\gamma \in \Gamma$, the associated \emph{shadow} is defined as $$S_{\eps}(\gamma):= \gamma.\left(M\setminus B_{\eps}(\gamma^{-1})\right)$$ where $B_{\eps}(\gamma^{-1})$ is the open ball of radius $\eps$ around $\gamma^{-1}$ with respect to the compactifying metric $d $ on $\Gamma \sqcup M$.
\end{defn}

\begin{fact}\cite[Proposition 5.1]{2024arXiv240409713B}\label{properties of shadows} Here are some properties of the shadows:
    \begin{itemize}
        \item For $\eps>0$, there exists $C'>0$ such that $$||\gamma||_{\sigma}-C'\leq \sigma(\gamma,\gamma^{-1}x)\leq ||\gamma||_{\sigma}+C'$$ for all $\gamma\in \Gamma$ and $x \in S_{\eps}(\gamma)$.
        \item If $\{g_n\} \subset \Gamma$ is an escaping sequence, then $$\lim_{n \rightarrow \infty } diam(S_{\eps}(g_n)) = 0 \text{ and } \lim_{n\rightarrow\infty} \inf_{x \in S_{\eps}(g_n)} d(g_n,x)=0.$$ In other words, the Hausdorff distance with respect to $d$ between $g_n$ and their shadows $S_{\eps}(g_n)$ converges to 0.
    \end{itemize}
\end{fact}

\begin{rmk}
    Classically the shadows are defined as projections of balls around $g.o$ to the boundary for some fixed base point $o$ in the associated metric space. Although $\Gamma \sqcup M$ is metrizable, the interior $\Gamma$ is equipped with the discrete topology, so we cannot naively uses the classical ways to construct shadows. Nevertheless, the shadows in Definition \ref{shadow} exhibits similar properties to the classical shadows. In fact, they are equivalent in Gromov hyperbolic spaces (see \cite[Proposition 5.5]{2024arXiv240409713B}).
\end{rmk}

\begin{lemma}\label{geod ball and shadows}
    Fix a compact set $K \subset \tilde{\Omega}_{\Gamma}$. For all $\eps >0$ sufficiently small, there exists $R>0$ large enough such that: if $t \geq R$, $v =(x,y,t) \in K$, and $\Tilde{\psi}^t(v) \in \gamma.\tilde{K}$ for some $\gamma \in \Gamma$, then $x$ is in the shadow $S_{\eps}(\gamma)$.
\end{lemma}

\begin{proof}
    We prove this by contradiction. Assume the lemma does not hold, then there exist $v_n, w_n \in K$ with $$v_n := (v_n^+, v_n^-, s_{1,n}) \text{ and }w_n := (w_n^+, w_n^-, s_{2,n})$$ and $t_n \rightarrow +\infty$ and $g_n \in \Gamma$ such that $$\Tilde{\psi}^{t_n}(v_n) = g_n .w_n \in g_n.K,$$ and $d(w_n^+, g_n^{-1}) \leq \frac{1}{n} < \eps, $ or equivalently $g_n. w_n^+ = v_n^+ \notin S_{\frac{1}{n}}(g_n)$ for $n$ large.

    Since $K$ is compact, up to taking subsequences, we can assume $v_n \rightarrow v:=(v^+, v^-, s_1) $ and $w_n \rightarrow w:= (w^+, w^-, s_2)$ in $ K. $ And because $t_n \rightarrow +\infty$, we have $g_n^{\pm1} \rightarrow \xi^{\pm} \in \Lambda_{\Gamma}$. Moreover, $d(w_n^+, w_n^-) $ is bounded away from zero as the flow line passes through $K$. Let $\eps_0: = \inf_n d(w_n^+, w_n^-)$, by triangle inequality, we have $d(w_n^-, g_n^{-1})>d(w_n^+,w_n^-)-d(w^+, g_n^{-1}) >\eps_0/2 $ for $n $ large enough. By assumption, $w_n^+ $ converges to $\xi^-$, then the convergence group action property ensures that $g_n. w_n^-= v_n^-$ converges to $v^-=\xi^+$. We also have $$d(g_n w_n^+, g_n)= d(v_n^+, g_n)\rightarrow d(v^+, \xi^+).$$
    Note that $v^+ \neq \xi^+$, because $v^- =\xi^+$ and $v^+\neq v^-$. This implies $d(v^+,\xi^+)$ is strictly bigger than 0, 
    then the fourth property in Fact \ref{prop of cocycle} implies \begin{equation}\label{4}
        \sigma'(g_n, w_n^+) \rightarrow -\infty .
    \end{equation}
    However, this gives a contradiction: by construction $s_{1,n}+t_n= s_{2,n}+\sigma'(g_n, w_n^+)$ and $s_{1,n}-s_{2,n}$ is bounded because $K$ is compact, then $t_n \rightarrow +\infty$ implies $\sigma'(g_n, w_n^+) \rightarrow +\infty$ which contradicts Equation \ref{4}, and hence proves the lemma.
\end{proof}

\begin{cor}\label{cor of geod ball and shadow}
For all $\eps>0$ small enough, there exists $R'>0$ such that, for all $g, g' \in \Gamma$, if there exist $v \in K$ and $t_{g'}>t_g>R'$ such that $\tilde{\psi}^{t_g}(v) \in g.K$ and $\tilde{\psi}^{t_{g'}}(v) \in g'.K$, we have $$g '\in g.\left((\Gamma\sqcup M\smallsetminus B_{\eps}(g^{-1})\right),$$ where $B_{\eps}({g}^{-1})$ is the open ball of radius $\eps$ around ${g}^{-1}$ defined by the compactifying metric $d$ on $\Gamma\sqcup M$.
\end{cor}

\begin{proof}
   The proof idea is similar to that of Lemma \ref{geod ball and shadows}. Assume this corollary fails, then there exist ${g}_n, g_n' \in \Gamma,$ $v_n \in K$, and $t_n':= t_{g'_n}>t_n:= t_{{g}_n} >n$ such that $$\tilde{\psi}^{t_n}(v_n) \in g_n.K, \tilde{\psi}^{t_n'}(v_n) \in g_n'.K \text{ and } d\left(({g}_n)^{-1}g'_n, ({g}_n)^{-1}\right) <\frac{1}{n}.$$
   Take a limit as $n \rightarrow \infty,$ we have $(g_n)^{-1}.K $ converges to some point $x \in \Lambda_{\Gamma}$ since $t_n \rightarrow \infty$ (equivalently $({g}_n)^{-1} \rightarrow x$). By assumption, $({g}_n)^{-1}g'_n$ also converges to $x$. Then the same argument in the proof of Lemma \ref{geod ball and shadows} gives rise to a contradiction.
\end{proof}

The shadows also satisfy a shadow lemma that estimates their measures with respect to the Patterson--Sullivan measure $\mu$ of dimension $\delta$ on $M.$

\begin{thm}[Shadow Lemma]\cite[Theorem 6.1]{2024arXiv240409713B}\label{shadow lemma}
    For $\eps>0$ small enough, there exists $c:= c(\eps)>1$ such that $$\frac{1}{c}\cdot\exp{(-\delta \cdot||\gamma||_{\sigma})}\leq \mu (S_{\eps}(\gamma))\leq c \cdot\exp{(-\delta \cdot||\gamma||_{\sigma}) }$$ for all $\gamma \in \Gamma.$
\end{thm}

Another use of the shadows is to give an alternative characterization of the set of conical limit points.

\begin{defn}[Uniform Conical Limit Points]
    Given $\eps>0$, the set of $\eps-$uniform limit points, denoted by $\Lambda_{\Gamma,\eps}^{con}$, consists $x \in M$ such that there exist $a,b \in M$ with $d(a,b)>\eps$ and a sequence $\gamma_n \subset \Gamma$ such that $$\lim_{n \rightarrow \infty} \gamma_n.x = a \text{ and } \lim_{n \rightarrow \infty} \gamma_n.y = b \text{ for all } y \in M \setminus \{x\}$$ 
\end{defn}

\begin{fact}\cite[Lemma 5.4]{2024arXiv240409713B}\label{alt char for conical} The sets $\Lambda_{\Gamma,\eps}^{com}$ have the following properties:
    \begin{itemize}
    \item $\Lambda_{\Gamma}^{con} = \cap_{\eps>0} \Lambda_{\Gamma,\eps}^{con}.$ 
        \item $\Lambda_{\Gamma,\eps}^{con}$ is $\Gamma-$invariant for all $\eps>0.$
        \item If $x \in \Lambda_{\Gamma,\eps}^{con} $ and $0<\eps'<\eps,$ then there exists an escaping sequence $\{\gamma_n\}$ such that $x \in \cap_n S_{\eps'}(\gamma_n).$
        \item If there exists an escaping sequence $\{\gamma_n\}$ such that $x \in \cap_n S_{\eps}(\gamma_n)$, then $x \in \Lambda_{\Gamma,\eps}^{con}.$
    \end{itemize}
\end{fact}

\section{Sufficient Condition for Finite BMS Measure}\label{finite bms}
The concept of positive recurrence originated from Sarig's work \cite{sarigthermo} and \cite{14ed99c1659848e68ee31e39133ee06d} on symbolic dynamics. This was later extended to the geodesic flow on the unit tangent bundle of negatively curve manifolds by Schapira--Pit \cite{2016arXiv161003255S}. In order to state the condition in our setting precisely, we need to define a couple of things first.

Let $M,\text{ }\Gamma$ be as defined in Section \ref{conv group action}. Fix two expanding cocyles $\sigma$ and $\sigma'$ with respect to $\Gamma$ and $M$, and $(\sigma, \bar{\sigma},G)$ a continuous GPS system. Let $\sigma'$ define the action of $\Gamma$ on $\Tilde{\Omega}_{\Gamma}$, and let $\Omega_{\sigma'}$ to denote the quotient of $\Tilde{\Omega}_{\Gamma}$ under this action as in Section \ref{flow space}. Let $\mu$ and $\bar{\mu}$ be the PS measures of dimension $\delta_{\sigma}$ on $M$ associated with $\sigma$ and $\bar{\sigma}$ respectively. Let $\tilde{m}_{\sigma}$ and $m_{\sigma}$ be the measures on $\tilde{\Omega}_{\Gamma}$ and $\Omega_{\sigma'}$ associated with the GPS system.

Let $K \subset \Tilde{\Omega}_{\Gamma}$ be a compact set. We define the \emph{subset of $\Gamma$ outside of $K$}, denoted by $\Gamma_K$, as follows: $\gamma \in \Gamma$ is in $\Gamma_K$ if there exist $v \in K$ and $t \geq 0$ such that $\Tilde{\psi}^t(v) \in \gamma.K$ and $\Tilde{\psi}^{[0,t]}(v) \cap \Gamma.K\subset  K\cup \gamma.K.$ Heuristically, the subset $\Gamma_K$ contains all group elements where there exists some flow line that passes through $K$ and $\gamma.K$ and it does not intersect any other copies of $K$ in between.

\begin{defn}[Positive Recurrence]\label{positive rec}
   We say the subgroup $\Gamma\subset Homeo(M)$ is \emph{positive recurrent} with respect to $\sigma', \sigma$, and a compact set $K \subset \Tilde{\Omega}_{\Gamma}$ if the Poincar\'e series $Q_{\Gamma,\sigma}$ diverges at the critical exponent $ \delta_{\sigma}$, and $$\sum_{\gamma \in \Gamma_{K} }||\gamma||_{\sigma'}  \cdot \exp{(-\delta_{\sigma} \cdot ||\gamma||_{\sigma})} <+\infty.$$
\end{defn}

An important tool used to prove that positive recurrence implies finite BMS measure on $\Omega_{\sigma'}$ is \emph{Kac's Lemma}. The original Kac's Lemma was stated for discrete dynamical systems, but one can circumvent this issue by discretizing the continuous flow $\psi^t$. 

 Let $X $ be a locally compact topological space, $\psi^t$ a continuous flow on $X$, and $\nu$ a Borel $\psi^t$ invariant measure on $X$. For any set $W \subset X$ and $\eps >0$, define $$W_{\eps}:= \bigcup_{t \in [0, \eps]} \psi^{-t}(W).$$ For any $x \in X$, define the \emph{$\eps-$hitting time} with respect to $W$ as \begin{equation}\label{16}
    \tau_{\eps, W}(x):= \inf \{t\geq \eps\text{ }|\text{ } \psi^t(x) \in W\}.
 \end{equation} 

\begin{prop}[Kac's Lemma for Flows]\cite[Proposition 4.2]{2016arXiv161003255S}\label{kac's lemma}
    Let $W$ be a compact subset of $X$ with positive $\nu$ measure. If the flow $\psi^t$ is ergodic and conservative with respect to $(X,\nu)$, then for every $\eps >0$, $W_{\eps}$ is also compact in $X$. Moreover, for almost all $x \in W_{\eps}$ with respect to $\nu$, the $\eps-$hitting time $\tau_{\eps, W}(x)$ is finite, and 
    \begin{equation}
        \sum_{k\in \N} k\cdot\nu\left(\{x \in W_{\eps}\text{ }|\text{ } \tau_{\eps, K}(x) \in [k\eps, (k+1)\eps]\}\right) = \nu(X).
    \end{equation}
\end{prop}

With this version of the Kac's Lemma, we can prove the following theorem.
\begin{thm}\label{pr implies finite}
     If there exists a compact set $\tilde{K} \subset \Tilde{\Omega}_{\Gamma}$ such that $\Gamma$ is positive recurrent with respect to $\tilde{K}$, then $m_{\sigma}(\Omega_{\sigma'}) <+\infty.$
\end{thm}

 We need the ergodicity of the flow $\psi^t$ on $\Omega_{\sigma'}$ with respect to the BMS measure $m_{\sigma}$ to prove the theorem.

\begin{prop}\label{erg of flow}
    If $\delta_{\sigma} <+\infty$ and $Q_{\Gamma,\sigma}(\delta_{\sigma})=+\infty$, then the flow $\psi^t$ on $\left(\Omega_{\sigma'}, m_{\sigma}\right)$ is conservative and ergodic.
\end{prop}

\begin{proof}
    Assume there exists a $\psi^t$-invariant set $A \subset \Omega_{\sigma'}$ with $m_{\sigma}(A) >0$ and $m_{\sigma}(A^c)>0$. After lifting we have $\Tilde{A}$ and $\Tilde{A^c} $ in $\Tilde{\Omega}_{\Gamma}  $ that are flow invariant and $\Gamma-$invariant. They can be written as $B \times \R$ and $B^c \times \R$ for some $B \subset \Lambda_{\Gamma}^{(2)}$ because of the flow invariance. Moreover, we have $$\tilde{m}_{\sigma}(\Tilde{A})  >0\text{ and } \tilde{m}_{\sigma}(\Tilde{A}^c) >0,$$ which implies $\mu\otimes \bar{\mu}(B)>0$ and  $\mu\otimes \bar{\mu}(B^c)>0$. However, by the dichotomy mentioned in Theorem \ref{dichotomy}, $\Gamma$ acts on $(M^{(2)}, \mu \otimes \bar{\mu})$ ergodically, which gives us a contradiction as $B$ is a $\Gamma$-invariant set in $\Lambda^{(2)} \subset M^{(2)}$ but $\mu\otimes\bar{\mu}(B)$ is neither 0 nor 1.
\end{proof}

The argument for the proof of Theorem \ref{pr implies finite} mainly follows the construction in Section 4 of \cite{2016arXiv161003255S}.
Taking $X = \Omega_{\sigma'}$, $\nu = m_{\sigma}$, by Proposition \ref{erg of flow}
the flow $\psi^t$ is ergodic and conservative on $(\Omega_{\sigma'}, m_{\sigma})$. Moreover, we can project $\tilde{K}$ to a compact set $K \subset \Omega_{\sigma'}$ with positive $m_{\sigma}$ measure, and apply the Kac's lemma to $(\Omega_{\sigma'}, m_{\sigma}, \psi^t)$ and $K$, and get \begin{equation}
    m_{\sigma}(\Omega_{\sigma'}) = \sum_{k \in \N} k \cdot m_{\sigma}(\{x \in K_{\eps}\text{ }|\text{ }\tau_{\eps, K}(x) \in [k\eps, (k+1)\eps]\}).
\end{equation}

    For any $\gamma \in \Gamma$, define  $$\mathcal{U}_{\tilde{K}, \gamma} := \{v \in \tilde{K} \text{ }|\text{ } \exists t \geq 0 \text{ with } \Tilde{\psi}^t(v) \in \gamma.\tilde{K}\}.$$ Lemma \ref{geod ball and shadows} essentially states that two copies of $\tilde{ K}$ that are ``far" enough are contained in the shadows of each other. We then have the following.

\begin{lemma}\label{bound geod ball by shadows}
   For any $\eps >0$ sufficiently small, let $R>0$ and $\gamma \in \Gamma$ be such that Lemma \ref{geod ball and shadows} holds, we have  $$\mathcal{U}_{\tilde{K}, \gamma} \subset S_{\eps}(\gamma)\times \gamma S_{\eps}(\gamma^{-1})\times I$$ for some bounded interval $I \subset \R$ independent of $\gamma$.
\end{lemma}
\begin{proof}
    For any $x \in \mathcal{U}_{\tilde{K},\gamma}$, let $x = (x^+,x^-,s).$ Lemma \ref{geod ball and shadows} implies $x^+ \in S_{\eps}(\gamma).$ Analogously, doing the same argument in the proof of Lemma \ref{geod ball and shadows} for $\gamma^{-1}\Tilde{\psi}^t(x) \in \mathcal{U}_{\tilde{K}, \gamma^{-1}}$ would show $\gamma^{-1} x^- \in S_{\eps}(\gamma^{-1})$. The third parameter in $\R$ is bounded because $K$ is compact.
\end{proof}

We now estimate the measure of $\mathcal{U}_{\tilde{K},\gamma}$.

\begin{lemma}\label{measure of geod ball}
    Let $\eps >0$ be sufficiently small, and $R$ be large enough such that Lemma \ref{geod ball and shadows} holds. There exists $C>0$ such that $$\tilde{m}_{\sigma}(\mathcal{U}_{\tilde{K},\gamma }) \leq C\cdot \exp{\left( -\delta_{\sigma} \cdot ||\gamma||_{\sigma}\right)} $$ for any $\gamma \in \Gamma$ such that there exists $v\in \mathcal{U}_{\tilde{K},\gamma}$ with $\Tilde{\psi}^t(v) \in \gamma \tilde{K}$ for some $t \geq R$. 
    \end{lemma}

    \begin{proof}
    By construction of the BMS measure $\tilde{m}_{\sigma}$, we have 
\begin{equation*}
\tilde{m}_{\sigma}\left(\mathcal{U}_{\tilde{K},\gamma} \right) \leq \int_{v= (v^+,v^-, s) \in \mathcal{U}_{\tilde{K},\gamma} } \exp{(\delta_{\sigma} \cdot  G(v^+, v^-))}d\mu\otimes d\bar{\mu} \otimes dt(v).
\end{equation*}

The function $G:M^{(2)} \rightarrow \R$ is bounded over the set $\{(v^+,v^-): v= (v^+,v^-, s) \in \tilde{K} \} \subset M^{(2)}$ because $\tilde{K}$ is compact. We can rewrite the equation above as \begin{equation}\label{5}
    \tilde{m}_{\sigma}\left(\mathcal{U}_{\tilde{K},\gamma} \right) \leq C'\cdot  \int_{v= (v^+,v^-, s) \in \mathcal{U}_{\tilde{K},\gamma} } d\mu\otimes d\bar{\mu} \otimes dt(v)
\end{equation} for some $C'>0$.
    
        According to Corollary \ref{bound geod ball by shadows}, Equation \ref{5} can be turned into        
        \begin{equation*}
            \tilde{m}_{\sigma}\left(\mathcal{U}_{\tilde{K},\gamma} \right) \leq C'\cdot  Leb(I) \cdot \mu\left(S_{\eps}(\gamma)\right)  \cdot \bar{\mu}\left(\gamma S_{\eps}(\gamma^{-1})\right)
        \end{equation*}
    

 Now $\bar{\mu}\left(\gamma S_{\eps}(\gamma^{-1})\right) \leq 1$ as $\bar{\mu}$ is a probability measure on $M$, and the Shadow Lemma (Theorem \ref{shadow lemma}) implies $\mu(S_{\eps}(\gamma)) \leq c\cdot \exp{(-\delta \cdot ||\gamma||_{\sigma})}$. Then we get $$\tilde{m}_{\sigma}\left(\mathcal{U}_{\tilde{K},\gamma} \right)\leq C\cdot \exp{(-\delta_{\sigma} \cdot ||\gamma||_{\sigma})}$$ with $C= c\cdot C'.$
\end{proof}

In order to make use of the Kac's Lemma to prove Theorem \ref{pr implies finite}, we need to examine the properties of the set $A_{K,k,\eps}:=\{x \in K_{\eps}\text{ }|\text{ }\tau_{\eps, K}(x) \in [k\eps, (k+1)\eps]\}$, where $\tau_{\eps, K}$ is as defined in Equation \ref{16}

\begin{prop}\label{est of magnitude}
    For all $\eps >0$ small enough, there exist $E> 0$ and a finite set $S \subset \Gamma$ such that the following hold.
   
        For all $k$ large enough, all $x\in K$ with $\tau_{\eps, K} (x) \in [k\eps,(k+1)\eps]$, and any lift $\Tilde{x} \in \tilde{K}$ of $x$, there exist $g,h \in S$ and $\gamma \in \Gamma_{K} $ such that $\Tilde{x} \in \mathcal{U}_{\tilde{K}, g\gamma h}$ and \begin{equation}\label{2}
            k\eps-E\leq ||\gamma||_{\sigma'}\leq (k+1)\eps +E.     
        \end{equation}
\end{prop}

\begin{proof}
    By definition of $\tau_{\eps,K}(x)$, there exists $T \in [k\eps, (k+1)\eps]$ such that $\psi^T(x)\in K$ and $\psi^s(x) \notin K$ for all $s \in [\eps, T-\eps]$. Lifting everything up to $\Tilde{\Omega}_{\Gamma}$ gives a flow line $\Tilde{\psi}^t(\Tilde{x})=:a_t$ that intersects $\Tilde{K}$ and $\gamma_0.\Tilde{K} $ for some $\gamma_0 \in \Gamma$, and $a_t|_{[\eps, T-\eps]}$ does not intersect $\Gamma.\Tilde{K}$. In particular, $\Tilde{x} \in \mathcal{U}_{\Tilde{K},\gamma_0}$.

    We now need to replace $\gamma_0$ by some $\gamma \in \Gamma_{\Tilde{K}}, $ where $\Gamma_{\Tilde{K}}$ is the subset of $\Gamma$ outside of $\Tilde{K} $ as defined before. Let $$\Tilde{K}_{[-\eps, \eps]} := \{\Tilde{x} \in \Tilde{\Omega}_{\Gamma} \text{ }|\text{ } \Tilde{\psi}^s(\Tilde{x}) \in \Tilde{K} \text{ for some } s \in [-\eps, \eps]\}.$$ Since $\Tilde{K}$ is compact and $\Gamma$ acts on $\Tilde{\Omega}_{\Gamma}$ properly discontinuously, the set $$S:= \{g \in \Gamma\text{ }|\text{ }g.\Tilde{K} \cap \Tilde{K}_{[-\eps, \eps]} \neq \emptyset \}$$ is finite. 

\begin{claim}
        There exist $g, h \in S$ such that $\gamma_0= g\gamma h$ for some $\gamma \in \Gamma_{\Tilde{K}}$
\end{claim}

\begin{claimproof}
    Let $$u:= \sup \{s\in [0, \eps]\text{ }|\text{ } a_s \in S.\Tilde{K}\}$$ and $$v := \inf  \{s\in [T-\eps, T]\text{ }|\text{ } a_s \in \gamma_0 S.\Tilde{K}\} .$$ One can then find $g_u, g_v \in S $ such that $a_u \in g_u.\Tilde{K}$ and $a_v \in \gamma_0 g_v.\Tilde{K}$. By the definition of $u$ and $v$, $a_t|_{[u,v]}$ does not intersect $\Gamma.\Tilde{K}$, this implies $\gamma := g_u^{-1}\gamma_0g_v \in \Gamma_{\Tilde{K}} $ and proves the claim. 
\begin{figure}[htp]
    \centering
    \includegraphics[clip,trim={5cm 8cm 5cm 5cm},width = 18cm]{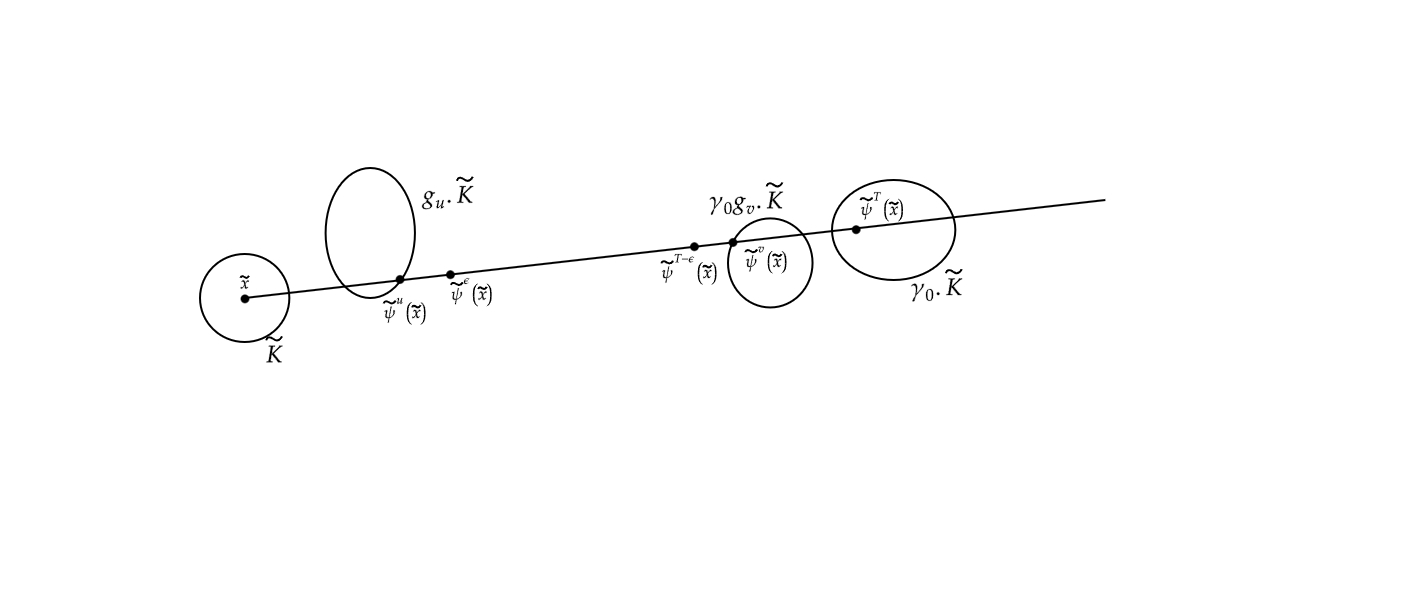}
\end{figure}
\end{claimproof}

    By construction $\Tilde{x} \in \mathcal{U}_{\Tilde{K},g_u\gamma g_v^{-1}}$. Then all that is left to prove Proposition \ref{est of magnitude} is to verify Equation \ref{2}. Since $S$ is finite, Fact \ref{prop of cocycle} implies that there exists $D$ such that \begin{equation}\label{3}
       ||\gamma_0||_{\sigma'}-2D \leq ||\gamma||_{\sigma'}\leq ||\gamma_0||_{\sigma'}+2D.
    \end{equation}

    Let $\Tilde{x}:= (\Tilde{x}^+,\Tilde{x}^-,t_0),$ and respectively $\Tilde{a}_t := (\Tilde{x}^+,\Tilde{x}^-,t_0+t).$ Moving $\Tilde{a}_T$ back to $\Tilde{K}$ by applying $\gamma_0^{-1}$ gives $$\gamma_0^{-1}.\Tilde{a}_T= \left(\gamma_0^{-1}\Tilde{x}^+,\gamma_0^{-1}\Tilde{x}^-, t_0+T+\sigma'(\gamma_0^{-1}, \Tilde{x}^+)\right) \in \Tilde{K}.$$
    Since $\Tilde{K}$ is a compact set in $\Tilde{\Omega}_{\Gamma}$, and $\Tilde{a}_0$ and $\gamma_0^{-1}\Tilde{a}_T$ are both in $\Tilde{K}$, $|t_0-\left(t_0+T+\sigma'(\gamma_0^{-1}, \Tilde{x}^+)\right)|=|T+\sigma'(\gamma_0^{-1}, \Tilde{x}^+)|= |T-\sigma'(\gamma_0, \gamma_0^{-1})\Tilde{x}^+|$ is bounded above by some $E'>0$.

Lemma \ref{geod ball and shadows} shows that $\Tilde{x}^+ \in S_{\eps}(\gamma_0)$. Let $C'>0 $ be the constant in Fact \ref{properties of shadows}. We then have $$||\gamma_0||_{\sigma'}-(E'+ C') \leq T\leq ||\gamma_0||_{\sigma'}+E'+ C'.$$ Equation \ref{3} then implies $$||\gamma||_{\sigma'}-(2D+E'+C') \leq T\leq ||\gamma||_{\sigma'}+2D+E'+ C'.$$ Furthermore, by construction, $T \in [k\eps, (k+1)\eps], $
and we have $$k \eps -(2D+E'+ C')\leq  ||\gamma||_{\sigma'} \leq (k+1)\eps+2D+E'+ C'.$$
Picking $E = 2D+E'+ C'$ gives us the desired inequalities to conclude the proof of Proposition \ref{est of magnitude}.
\end{proof}

\begin{proof}[Proof of Theorem \ref{pr implies finite}]
    We can now prove the theorem. For all $x \in A_{K,k, \eps}$ with large enough $k$, there is some $s \in [0,\eps]$ such that $x_s:=\psi^{s}(x) \in K$. We have $\tau_{\eps,K}(x_s) \in [(k-1)\eps, (k+1)\eps]$, and  it has a lift $\Tilde{x}_s$ in $\Tilde{K}$. Proposition \ref{est of magnitude} ensures that there exists a finite set $S$ such that $\Tilde{x}_s \in \mathcal{U}_{\Tilde{K}, g\gamma h}$ for some $\gamma \in \Gamma_{\Tilde{K}}$ and $g,h \in S$. This means for any lift $\Tilde{x}$ of $x$ into $\Tilde{K}_{\eps}$ satisfies \begin{equation}\label{6}
        \Tilde{x} \in \bigcup_{\gamma \in \Gamma_{\Tilde{K},k}} \bigcup_{g,h \in S} \bigcup_{s \in [0,\eps]} \Tilde{\psi}^s\left(\mathcal{U}_{\Tilde{K},g\gamma h}\right)
    \end{equation}
where $\Gamma_{\Tilde{K},k}:= \{g \in \Gamma_{\Tilde{K}} \text{ }|\text{ } (k-1)\eps -E\leq ||g||_{\sigma'} \leq (k+1)\eps +E\}$ for $E$ in Proposition \ref{est of magnitude}.

Then the term $m_{\sigma}(A_{K,k,\eps})$ that appears in the left hand side of the Kac's lemma can be rewritten as \begin{equation*}
    m_{\sigma}\left(A_{K,k,\eps}\right) \leq \sum_{g,h \in S } \sum_{\gamma \in \Gamma_{\Tilde{K},k}} \Tilde{m}_{\sigma}\left(\bigcup_{s\in[0,\eps]}\Tilde{\psi}^s\left(\mathcal{U}_{\Tilde{K},g\gamma h}\right) \right)
\end{equation*}
using Equation \ref{6}. Moreover, $\bigcup_{s\in[0,\eps]}\Tilde{\psi}^s\left(\mathcal{U}_{\Tilde{K},g\gamma h}\right)$ is contained in $S_{\eps}(g\gamma h)\times S_{\eps}\left((g\gamma h)^{-1}\right) \times I'$ with $I'$ a slightly enlarged interval than $I$ as in Corollary \ref{bound geod ball by shadows}. Then using Lemma \ref{measure of geod ball}, there exists $C>0$ such that $$m_{\sigma}\left(A_{K,k,\eps}\right) \leq C \cdot  \sum_{g,h \in S } \sum_{\gamma \in \Gamma_{\Tilde{K},k}} \exp{\left(-\delta_{\sigma} \cdot ||g \gamma h||_{\sigma}\right)} .$$ 

Since $S$ is finite, Fact \ref{prop of cocycle} again implies that $||\gamma||_{\sigma}$ is bounded away from $||g \gamma h||_{\sigma}$. Moreover, we have
\begin{equation}\label{7}
    m_{\sigma}\left(A_{K,k,\eps}\right) \leq C' \cdot  \sum_{\gamma \in \Gamma_{\Tilde{K},k}} \exp{\left(-\delta_{\sigma} \cdot || \gamma ||_{\sigma}\right)}
\end{equation}
    for some $C'>0$.

Using Equation \ref{7} we can rewrite the sum in Kac's Lemma as $$m_{\sigma}\left(\Omega_{\sigma'}\right)=\sum_{k\geq 0}k \cdot m_{\sigma}(A_{K,k,\eps})\leq  C_0 + C'\cdot \sum_{k \geq k_0} k\cdot\sum_{\gamma \in \Gamma_{\Tilde{K},k}} \exp{\left(-\delta_{\sigma} \cdot || \gamma ||_{\sigma}\right)}$$ where for $k\geq k_0$ all of the above arguments hold and $C_0$ the sum of the initial part of the series when $k <k_0$.

By definition of $\Gamma_{\Tilde{K},k}$, we have $(k-1)\eps-E \leq ||\gamma||_{\sigma'}$ for all $\gamma \in \Gamma_{\Tilde{K},k}.$ We can again pick $ k_0'$ large enough such that $(k-1)\eps-R >\frac{k\eps}{2}$ for all $k \geq k_0'$, then the equation above becomes \begin{equation}\label{8}
    m_{\sigma}\left(\Omega_{\sigma'}\right)= \sum_{k\geq 0}k \cdot m_{\sigma}(A_{K,k,\eps})\leq  C_0' + \frac{2C'}{\eps}\cdot \sum_{k \geq k_0'} \sum_{\gamma \in \Gamma_{\Tilde{K},k}} ||\gamma||_{\sigma'}\cdot \exp{\left(-\delta_{\sigma} \cdot || \gamma ||_{\sigma}\right)}.
\end{equation}

For any $\gamma \in \Gamma_{\Tilde{K}}$, the element $\gamma $ is in $\Gamma_{\Tilde{K},k}$ if and only if $$\frac{||\gamma||_{\sigma'}-E-\eps}{\eps} \leq k \leq \frac{||\gamma||_{\sigma'}+E+\eps}{\eps},$$ i.e. there are at most $\frac{2(E+\eps)}{\eps}$ many choices for $k$ such that $\gamma \in \Gamma_{\Tilde{K},k}$. Then Equation \ref{8} becomes $$m_{\sigma}\left(\Omega_{\sigma'}\right)= \sum_{k\geq 0}k \cdot m_{\sigma}(A_{K,k,\eps})\leq  C_0' +C'' \cdot \sum_{\gamma \in \Gamma_{\Tilde{K}} } ||\gamma||_{\sigma'}\cdot \exp{\left(-\delta_{\sigma} \cdot || \gamma ||_{\sigma}\right)}$$ for $C'':= \frac{4C'(E+\eps)}{\eps^2}$. The positive recurrence condition ensures the sum in the above inequality is finite, and hence proves that $m_{\sigma}\left(\Omega_{\sigma'}\right)$ is finite.
\end{proof}

\section{Strongly Positive Recurrence}\label{spr}

In this section, we give a definition of the property that we call \emph{strongly positive recurrent} for a convergence group $\Gamma \subset Homeo(M)$. Let $\Tilde{\Omega}_{\Gamma}$ be the flow space as defined in Section \ref{flow space}, $(\sigma, \bar{\sigma}, G)$ be a GPS system associated with the pair $(M, \Gamma)$. Fix an expanding continuous cocycle $\sigma'$, and throughout this section we consider the $\Gamma$ action on $\Tilde{\Omega}_{\Gamma}$ induced by $\sigma'$.

\subsection{Subset Outside of A Given Compact Set}\label{Gamma_K}

For any compact set $K \subset \Tilde{\Omega}_{\Gamma}$, let $\Gamma_K$ denote the subset of $\Gamma$ outside of $K$ as defined in Section \ref{finite bms}. We want to show that given two ``large" enough compact sets $K_1, K_2 \subset \Tilde{\Omega}_{\Gamma}$, the sets $\Gamma_{K_1}$ and $\Gamma_{K_2}$ are related in a nice way. 

In order to make the large enough condition of the compact sets precise, we need the following property of the convergence group $\Gamma$.
\begin{fact}[see e.g. Lemma 2.4 \cite{2024arXiv240409713B}]\label{approx loxo}
    There exist $\eps_0 >0$ and a finite set $F $ such that for any $\gamma \in \Gamma$, there exists $f \in F$ such that $\gamma f$ is loxodromic and $$\min \{d\left((\gamma f)^+, (\gamma f)^-\right), d\left(\gamma f, (\gamma f)^-\right), d\left(\gamma f, (\gamma f)^+\right)\}>\eps_0,$$ where $d$ is the compactifying metric on $\Gamma \sqcup M$ as defined in Section \ref{cont cocycle} and $(\gamma f)^{\pm}$ are the fixed points of $\gamma f$.
\end{fact}

With this fact, we define $$\delta_0:= \min \{d(fx,fy) \text{ }|\text{ } x,y \in M,\text{ } d(x,y) \geq \eps_0, \text{ and } f \in F\},$$ and $K_0:= K_0' \times \{0\} \subset \Tilde{\Omega}_{\Gamma},$ where $$K_0':= \{ (x,y)\in M^{(2)} \text{ }| \text{ } d(x,y) \geq \min\{\delta_0, \eps_0\}\}.$$

The next lemma shows that for any pair of translates of $K_0$, there always exists some flow line that intersects both.

\begin{lemma}\label{flow line through translates}
    For all $g \in \Gamma$, there exists $v \in K_0$ such that the flow line $\tilde{\psi}^{[0,\infty)}(v)$ intersects $g.K_0$.
\end{lemma}
\begin{proof}
    We just need to show that $K_0' \cap g.K_0' \neq \emptyset$ for all $g\in \Gamma.$ If the intersection is non-empty, we can find $(v_g^+, v_g^-)\in K_0' \cap g.K_0'$. By construction $(v_g^+, v_g^-, 0) \in K_0$, and $\left(g^{-1}.v_g^+,g^{-1}.v_g^-, 0\right) \in K_0 $. Then $$g. \left(g^{-1}.v_g^+,g^{-1}.v_g^-, 0\right) = \left(v_g^+, v_g^-, \sigma'(g,g^{-1}.v_g^+)\right) \in g.K_0,$$ i.e. the flow line defined by $v = (v_g^+, v_g^-, 0)$ passes through $K_0$ and $g.K_0$. 
    
    We now prove the claim. Fix $g \in \Gamma$. Let $f\in F$ be an element such that $gf$ is loxodromic and let $(gf)^{\pm} \in M$ be the distinct fixed points of $gf$ as in Fact \ref{approx loxo}. We have $d\left((g f)^+,(gf)^-\right) > \eps_0,$ while $d\left(f.(gf)^+, f.(gf)^-\right)\geq \delta_0 $ by construction. This implies that $\left((gf)^+, (gf)^-\right)$ and $f.\left((gf)^+, (gf)^-\right) $ are both in $K_0'$ and $$\left((gf)^+, (gf)^-\right) =(gf).\left((gf)^+, (gf)^-\right) =g.\left(f.\left((gf)^+, (gf)^-\right)\right)=  \in g.K_0'.$$ This proves the claim, and hence the lemma. 
\end{proof}

\begin{prop}\label{containment of Gamma_K}
    Let $K_1, K_2$ be two compact sets in $\Tilde{\Omega}_{\Gamma}$ that both contain $K_0$. If $K_1 \subset int(K_2)$, then there exists a finite set $S:= S(K_1, K_2) \subset \Gamma$ such that $$\Gamma_{K_2} \cap \left(\Gamma\smallsetminus S\right) \subset \bigcup_{g,h \in S} g \Gamma_{K_1} h^{-1}.$$ 
\end{prop}

\begin{proof}
    Let $L$ and $R$ be the constants defined in Lemma \ref{disjoint flow lines}. Define $$S':= \{\gamma \in \Gamma \text{ }|\text{ } \gamma.K_1 \cap (K_1)_{[-R,R]}\}$$ where $ (K_1)_{[-R,R]}:= \{v \in \Tilde{\Omega}_{\Gamma} \text{ }|\text{ } \Tilde{\psi}^s(v) \in K_1 \text{ for some } s \in [-R,R]\}$.

    Choose $g \in \Gamma_{K_2}$ such that $t_{g,K_2}> L$. If no such $g$ exists, then $\Gamma_{K_2}$ is a finite set and the proposition hold trivially. For any $v \in K_1$, if there exists $t_v>0$ such that $\Tilde{\psi}^{t_v}(v) \in g.K_1$, Lemma \ref{disjoint flow lines} implies $\Tilde{\psi}^{[R, t_v-R]} (v)\cap \Gamma.K_1 = \emptyset$. Using a similar argument in the proof of Proposition \ref{est of magnitude}, we can find $h_1, h_2  \in S'$ such that the flow line $\tilde{\psi}^{[0,\infty)}(v)$ does not intersect any translates of $K_1$ between $h_1.K_1$ and $gh_2.K_1$, i.e. $h_1^{-1}gh_2\in \Gamma_{K_1}.$ Taking $S$ to be the union of $S'$ and the finite set of elements $g \in \Gamma_{K_1}$ with $t_{g,K_2}<L$ will prove the proposition is proved.
\end{proof}

\subsection{Entropy at Infinity}\label{entropy at infinity}


Given a compact set $K \subset \Tilde{\Omega}_{\Gamma}$, we can define a Poincar\'e series associated with $\Gamma_{K}$ as follows: $$Q_{\Gamma_{K}, \sigma}(s):= \sum_{\gamma\in \Gamma_{K}} \exp{\left(-s \cdot ||\gamma||_{\sigma}\right)}.$$ Let $\delta_{\sigma}(\Gamma_K):= \inf \{s >0\text{ }|\text{ } Q_{K^c, \sigma}(s)<+\infty\}$ be the associated critical exponent. 

\begin{lemma}\label{decreasing entropy}
    If $K_1, K_2$ are compact sets in $\Tilde{\Omega}_{\Gamma}$ that contain $K_0$ and $K_1 \subset int(K_2),$ then $\delta_{\sigma}(\Gamma_{K_2}) \leq \delta_{\sigma}(\Gamma_{K_1}).$
    
\end{lemma}

\begin{proof}
    Lemma \ref{containment of Gamma_K} implies there exists a finite set $S$ such that $$\left( \Gamma_{K_2} \cap \left(\Gamma\smallsetminus S\right) \right)\subset \bigcup_{g,h \in S} g \Gamma
    _{K_1} h^{-1}.$$ We can then rewrite the Poincar\'e series of $\Gamma_{K_2}$ as 
    \begin{equation}
    \begin{split}
        \sum_{\gamma \in \Gamma_{K_2}} \exp{\left(-s\cdot ||\gamma||_{\sigma}\right)} &\leq \sum_{g,h \in S}\sum_{\gamma \in \Gamma_{K_1}} \exp{\left(-s\cdot ||g\gamma h^{-1}||_{\sigma}\right)} \\
        & \leq \sum_{g,h \in S} \sum_{\gamma \in \Gamma_{K_1}} C\cdot \exp{\left(-s\cdot ||\gamma||_{\sigma}\right)} \text{ \quad by Fact \ref{prop of cocycle}}\\
        & = 2\cdot|S|\cdot C \sum_{\gamma \in \Gamma_{K_1}} \exp{\left(-s\cdot ||\gamma||_{\sigma}\right)}.
    \end{split}
    \end{equation}
    This means $Q_{K_2,\sigma} $ converges if $Q_{K_1,\sigma}$ converges, hence $\delta_{\sigma}(\Gamma_{K_2}) \leq \delta_{\sigma}(\Gamma_{K_1}).$    
\end{proof}

\begin{defn}
    We define the \emph{entropy at infinity} for $\Gamma$, denoted by $\delta_{\infty}(\Gamma)$, as $$\delta_{\infty}(\Gamma) := \inf_{K \subset \Tilde{\Omega}_{\Gamma} \text{ compact}, \text{ } K_0 \subset K } \delta_{\sigma}(\Gamma_{K}).$$
\end{defn}

\begin{rmk}\label{rewrite delta infty}
    Due to Lemma \ref{decreasing entropy}, we can rewrite the entropy at infinity for $\Gamma$ as $$\delta_{\infty}(\Gamma) = \lim_{i\rightarrow \infty} \delta_{\sigma}(\Gamma_{K_i})$$ for any exhausting sequence of compact sets $K_i$ in $\Tilde{\Omega}_{\Gamma}$ such that $K_0 \subset K_1,$ and $K_i \subset int(K_{i+1})$.
\end{rmk}

\subsection{Strongly Positively Recurrent Subgroups} With the discussion above, we can define the following. 

\begin{defn}
   The group $\Gamma \subset Homeo(M)$ is \emph{strongly positively recurrent} (SPR) with respect to $\sigma$ and $\sigma'$ if $\delta_{\infty}(\Gamma) <\delta := \delta_{\sigma}(\Gamma)$.
\end{defn}

The SPR condition is also referred as a \emph{critical gap at infinity} by Schapira--Tapie in \cite{2018arXiv180204991S}. We will show that if $\Gamma$ is strongly positively recurrent, then the induced BMS measure $m_{\Gamma}$ on $\Omega_{\Gamma}$ is finite.

\begin{thm}\label{spr implies finite BMS}
    Assume $\Gamma$ is SPR with respect to $\sigma$ and $\sigma'$ and the magnitudes induced by the cocycles satisfy the following: for all $\eps>0$, there exists $C>0$ such that 
    \begin{equation}\label{relate magnitudes}
        ||\gamma||_{\sigma'} \leq C\cdot \exp{\left(\eps \cdot ||\gamma||_{\sigma}\right)}
    \end{equation}
    for all $\gamma \in \Gamma$,
    then $m_{\Gamma}(\Omega_{\Gamma}) <+\infty$ for the BMS measure $m_{\Gamma}$ induced by the GPS system $(\sigma, \bar{\sigma}, G).$
\end{thm}

The proof idea for Theorem \ref{spr implies finite BMS} is mainly inspired by Schapira--Tapie \cite{2018arXiv180204991S}. To do this, we employ Theorem \ref{pr implies finite} that we proved earlier. Essentially, we just need to show the SPR condition implies the Poincar\'e series $Q_{\Gamma,\sigma}$ diverges at the critical exponent $\delta.$ 

\begin{thm}\label{spr implies divergent}
    If $\Gamma$ is SPR with respect to $\sigma$ and $\sigma'$, then $Q_{\Gamma, \sigma}$ diverges at its critical exponent $\delta$.
\end{thm}

Assume this thereom holds. Then the SPR condition and Remark \ref{rewrite delta infty} imply that there exists a large enough compact set $K \subset \Tilde{\Omega}_{\Gamma}$ such that $\delta_{\sigma}(\Gamma_K) <\delta$. Then we have
\begin{equation}\label{9} 
\sum_{g \in \Gamma_{K}} ||g||_{\sigma'}\exp{\left(-\delta\cdot ||g||_{\sigma}\right)} = \sum_{g \in \Gamma_{K}} \exp{\left(-\delta\cdot ||g||_{\sigma} + \log (||g||_{\sigma'})\right)} .
\end{equation}


  
  Condition \ref{relate magnitudes} on the magnitudes induced by the two cocycles allows us to rewrite Equation \ref{9} such that for every $\eps>0$, there exist $A, C$ where 
    \begin{equation}
        \sum_{g \in \Gamma_{K}} ||g||_{\sigma'}\exp{\left(-\delta\cdot ||g||_{\sigma}\right)} \leq A+ C\cdot \sum_{g \in \Gamma_{K}} \exp{\left(-(\delta-\eps)\cdot ||g||_{\sigma} \right)}.
    \end{equation}
    
Since $\delta_{\sigma}(\Gamma_K)$ is strictly smaller than $\delta$, for $\eps$ small enough we have the second term in the last line of the inequalities above converges. This in turn gives $\sum_{g \in \Gamma_{\Tilde{K}}} ||g||_{\sigma'}\exp{\left(-\delta\cdot ||g||_{\sigma}\right)} <+\infty$, which means $\Gamma$ is positively recurrent in the sense of Definition \ref{positive rec}. By Theorem \ref{pr implies finite}, we have $m_{\Gamma}(\Omega_{\Gamma}) < +\infty,$ and this concludes the proof of Theorem \ref{spr implies finite BMS}. 

Now all that remains is to prove Theorem \ref{spr implies divergent}. To do this, we employ the Hopf--Tsuji--Sullivan type dichotomy mentioned in Theorem \ref{dichotomy}. The goal is to show the conical limit set has full measure with respect to some Patterson--Sullivan measure $\mu$ of dimension $\delta$, and the dichotomy would imply the Poincar\'e series $Q_{\Gamma, \sigma}(\delta) = +\infty.$ We first describe a set that contains the complement of the conical limit set.

\begin{defn}
     For any compact set $ K \supset K_0$ in $\tilde{\Omega}_{\Gamma}$, we define:
     \begin{itemize}
         \item $\Lambda_{K^c}: = \{\xi \in \Lambda_{\Gamma} \text{ }|\text{ } \exists v =(\xi,\xi_-,t_v) \in K \text{ with } \tilde{\psi}^{[0,\infty)}(v) \cap \Gamma.K \subset K\};$
         \item $\mathcal{L}_{K^c}:= \Gamma. \Lambda_{K^c}.$
     \end{itemize}
\end{defn}

The set $\Lambda_{K^c}$ denotes the set of forward endpoints of flow lines that leave the compact set $K$ and never return to any translates of $K$, and $\mathcal{L}_{K^c}$ denotes the set of points in the limit set that are the forward limit points of some flow line that eventually leaves $\Gamma.K$ as shown by the following lemma. 

\begin{lemma}\label{limit set outside of K}Here are some properties of $\mathcal{L}_{K^c}:$
\begin{enumerate}
    \item If $\xi \in  \mathcal{L}_{K^c}$, then there exist $ v=(\xi,\xi_-,t_v) \in \Gamma.K $ and $T\geq 0$ such that $  \Tilde{\psi}^{[T,\infty)}(v)$ does not intersect $\Gamma.K.$

    \item If $K_1 \subset K_2 $ are compact sets in $\tilde{\Omega}_{\Gamma}$ that contain $K_0$, then $\mathcal{L}_{K_2^c} \subset \mathcal{L}_{K_1^c}.$

    \item For any compact set $K\supset K_0$, $\Lambda_{\Gamma} \setminus \mathcal{L}_{K^c} \subset \Lambda_{\Gamma}^{con}$.
\end{enumerate}
\end{lemma}

\begin{proof}
   For all $\xi \in \mathcal{L}_{K^c}$, there exist $g \in \Gamma$ and $\xi' \in \Lambda_{K^c}$ such that $g.\xi'= \xi$. By definition of $\Lambda_{K^c}$, there exists $v=(\xi', \xi'_-,t_v) \in K$ such that $\Tilde{\psi}^{[0,\infty)}(v) \cap \Gamma.K \subset K $. Note that $$g.v= (g.\xi',g.\xi'_-,t_v+\sigma'(g, \xi')).$$ Since the group action commutes with the flow $\tilde{\psi}^s$, the flow line that starts at $g.v$ only intersects $g.K$. Then $g.v$ is the desired element in $\tilde{\Omega}_{\Gamma}$ to prove $(1)$.

To prove part $(2)$, let $\xi \in \mathcal{L}_{K_2^c}$. Part $(1)$ shows that there exist $g \in \Gamma$, $T >0$ and $v \in g.K_2$ such that $\tilde{\psi}^{[T,\infty)}(v) \cap \Gamma.K_2 = \emptyset.$ Pick $v' \in g.K_1$ such that $\tilde{\psi}^{\infty}(v') = \xi$. Corollary \ref{flow lines with same endpoint are close} implies that there exists $L' >0$ such that $\tilde{\psi}^{[T+L', \infty)}(v') \cap \Gamma.K_1 = \emptyset$, hence $\xi \in \mathcal{L}_{K_1^c}.$

  In order to prove part $(3)$, one recalls the alternative characterization of the conical limit set in Fact \ref{alt char for conical}. If $\xi \in \Lambda_{\Gamma} \setminus \mathcal{L}_{K^c}$, then for all $ v \in \Gamma.K$ with $\tilde{\psi}^{\infty}(v) = \xi$, the flow line $\tilde{\psi}^{[D,\infty)}(v)$, with $D$ being the maximum time it takes to flow out of $K$, will intersect $g.K$ for some $g \in \Gamma$. This implies that $\tilde{\psi}^{[0,\infty)}(v)$ intersects $\Gamma.K$ infinitely many times. Let $\{\gamma_n\}$ denote the elements in $\Gamma$ such that $\gamma_n.K$ intersects $\tilde{\psi}^{[0,\infty)}(v) $. Lemma \ref{geod ball and shadows} states that for $\gamma_n.K$ far  enough from $K$, $x \in S_{\eps}(\gamma_n)$ for some fixed $\eps>0$. Fact \ref{alt char for conical} then implies that $\xi \in \Lambda_{\Gamma,\eps}^{con}\subset \Lambda_{\Gamma}^{con},$ and this concludes the proof.
\end{proof}

 Fix compact sets $K_1, K_2$ in $\Tilde{\Omega}_{\Gamma}$ such that $K_0 \subset K_1 \subset int(K_2)$ and $\delta_{\sigma}(\Gamma_{K_2})\leq \delta_{\sigma}(\Gamma_{K_1})<\delta$ (the existence of which is guaranteed by the SPR condition).

\begin{lemma}\label{flow lines with same endpoints need to avoid the same set}
    There exists a finite set $F:= F(K_1, K_2) \subset \Gamma$ such that for all $v , v'\in K_2$ and $g \in \Gamma\smallsetminus F$, if $ \tilde{\psi}^{\infty}(v)= \tilde{ \psi}^{\infty}(v')$ and $\tilde{\psi}^{[0,\infty)}(v) \cap g.K_2 =\emptyset$, then $\tilde{\psi}^{[0,\infty)}(v') \cap g.K_1 = \emptyset.$
\end{lemma}

\begin{proof}
    Define $$F: = \left\{ h \in \Gamma\text{ }|\text{ } h.K_1 \cap \left(\bigcup_{s \in [-L,L]}\tilde{\psi}^s(K_2)\right) \neq \emptyset \right\}.$$ Then Corollary \ref{flow lines with same endpoint are close} implies that if $\tilde{\psi}^s(v') \cap g.K_1 \neq \emptyset$ for some $s$ and $g \in \Gamma \smallsetminus F$, then the flow line $\tilde{\psi}^s(v)$ has to intersect $ g.K_2$, which gives a contradiction.
\end{proof}

Lemma \ref{flow lines with same endpoints need to avoid the same set} provides a way to describe points ``near" the set $\Lambda_{K_2^c}$. In order to make the definition of neighborhoods around $\Lambda_{K_2^c}$ precise, we need to describe the topology on $\tilde{\Omega}_{\Gamma} \sqcup \Lambda_{\Gamma}.$
  
As mentioned before, the flow space $\tilde{\Omega}_{\Gamma}$ admits a compactifying topology (we refer the readers to Section 7.1 in \cite{2024arXiv240409718B} for a detailed description of the basis of this topology), and its closure is $\tilde{\Omega}_{\Gamma} \sqcup \Lambda_{\Gamma}$. Furthermore, by fixing $o \in K_0 \subset  \tilde{\Omega}_{\Gamma}$, we can identify $\Gamma \sqcup \Lambda_{\Gamma}$ with the closure of its orbit $\overline{\Gamma.o} \subset \tilde{\Omega}_{\Gamma}\sqcup \Lambda_{\Gamma}$. It is not hard to see that by construction the subspace topology restricted to $(\Gamma.o\sqcup \Lambda_{\Gamma} )$ induces a compactifying topology on $\Gamma\sqcup \Lambda_{\Gamma}$. This is equivalent to the topology on $\Gamma \sqcup \Lambda_{\Gamma}$ induced by the metric $d$ defined in Section \ref{cont cocycle} \cite[Proposition 2.3]{2024arXiv240409713B}. For the rest of the section, we will not distinguish between $\Gamma \sqcup \Lambda_{\Gamma}$ and $\Gamma.o \sqcup \Lambda_{\Gamma}$ when the context is clear.

\begin{defn}\label{U_T} 
    We define the set $U_T(K_1, F)\subset \Gamma\sqcup \Lambda_{\Gamma}$ as the set of points $x$ such that \begin{itemize}
        \item  $x \in \Gamma\cap U_T(K_1, F)$ if and only if there exists $v \in K_1$ such that $$x.K_1 \cap \Tilde{\psi}^{[0,\infty)}(v)  \neq \emptyset \text{ and } \Tilde{\psi}^{[0,T]}(v) \cap \Gamma.K_1 \subset F.K_1;$$
        \item $x \in \Lambda_{\Gamma} \cap U_T(K_1, F)$ if and only if there exists $v\in K_1$ such that $$\tilde{\psi }^{\infty}(v) = x \text{ and } \Tilde{\psi}^{[0,T]}(v) \cap \Gamma.K_1 \subset F.K_1 .$$
    \end{itemize}
\end{defn}

\begin{lemma}\label{nbhd of Lambda_K2c}
    For all $\xi \in \Lambda_{K_2^c}$ and $T>0$, there exists an open neighborhood $I_{\xi, T}$ of $\xi$ in $\tilde{\Omega}_{\Gamma}\sqcup \Lambda_{\Gamma}$ such that $I_{\xi, T} \subset U_T(K_1,F)$.
\end{lemma}

\begin{proof}
    If not, then for some $\xi\in \Lambda_{K_2^c}$ there exist either (1) $\{x_n\} \subset \Lambda_{\Gamma}$ such that $x_n \rightarrow \xi$ and $x_n \notin U_T(K_1,F),$ or 2) $\{g_n\} \subset \Gamma $ such that $g_n.o \rightarrow \xi$ and $g_n \notin U_T(K_1,F).$

    In the first case, there exist $v_n \in K_1, \gamma_n \in \Gamma \smallsetminus F$ such that $$\tilde{\psi}^{\infty}(v_n) = x_n, \text{ and } \tilde{\psi}^{[0,T]}(v_n) \cap \gamma_n.K_1 \neq \emptyset.$$
    Since $K_1$ is compact, so is the set $\bigcup_{s\in [0,T]}\tilde{\psi}^s(K_1).$ Taking limit as $n \rightarrow \infty$, up to passing to a subsequence, we have $v_n \rightarrow v \in K_1$ and $\gamma_n \rightarrow \gamma \in \Gamma\smallsetminus F$. Since $\tilde{\psi}^{[0,T]}(v_n)$ intersects $\gamma_n.K_1$, the limit $\tilde{\psi}^{[0,T]}(v)$ also intersects $\gamma.K_1$. However, this cannot happen as $\tilde{\psi}^{\infty}(v) = \xi \in \Lambda_{K_2^c}$. By definition of $\Lambda_{K_2^c}$, there exists $v' \in K_2$ such that $\tilde{\psi}^{\infty}(v') = \xi,$ and $\tilde{\psi}^{[0,\infty)}(v') \cap \Gamma.K_2 \subset K_2$. Then Corollary \ref{flow lines with same endpoints need to avoid the same set} implies $\tilde{\psi}^{[0,\infty)}(v) \cap \gamma.K_1 = \emptyset$ and gives a contradiction.

    For the second case, there exist $v_n \in K_1$, $t_n \geq T$, and $\gamma_n \in \Gamma\smallsetminus F$ such that $$\tilde{\psi}^{t_n}(v_n) \in  g_n.K_1 , \text{ and } \tilde{\psi}^{[0,T]}(v_n) \cap \gamma_n.K_1 \neq \emptyset.$$ The same limiting argument as in the prior case shows that there exist $v \in K_1$ and $\gamma \in \Gamma \smallsetminus F$ such that $\tilde{\psi}^{\infty}(v)= \xi$ but $\tilde{\psi}^{[0,T]}(v) \cap \gamma.K_1 \neq \emptyset$, and we again have a contradiction. 
\end{proof}

\begin{prop}\label{measure of U_T}
             Define $\mathcal{U}_T(\Lambda_{K_2^c}): = int(U_T(K_1,F))$. We have  $$\lim_{T\rightarrow \infty} \mu\left(\mathcal{U}_T(\Lambda_{K_2^c})\right) = 0,$$ where $\mu$ is a Patterson--Sullivan measure of dimension $\delta$ supported on $\Lambda_{\Gamma}.$
\end{prop}

\begin{proof}
     Since we do not have $Q_{\Gamma,\sigma}(\delta) = +\infty$ apriori, we will look at the modified Poincar\'e series $$\Tilde{Q}_{\Gamma,\sigma}(s):= \sum_{\gamma\in \Gamma} \chi\left(||\gamma||_{\sigma} \right)\cdot \exp{\left(-s\cdot ||\gamma||_{\sigma}\right)}$$ where $\chi: \R \rightarrow \R_{\geq 1}$ is a non-decreasing function such that for all $\eps >0$ there exists $R>0$ with $\chi(r+t)\leq \exp{(\eps t)}\cdot \chi(r)$ for all $r\geq R$ and $t\geq 0.$  This construction originates from Patterson \cite{10.1007/BF02392046}, where he showed that there exists such function $h$ that makes the series $\Tilde{Q}_{\Gamma,\sigma}(s)$ diverges at $s= \delta $ and converges for all $s >\delta.$

    One can construct a sequence of Borel probability measures on $\Gamma \sqcup M$ for each $s > \delta$: $$\mu_s:= \frac{1}{\Tilde{Q}_{\Gamma,\sigma}(s)} \sum_{\gamma\in \Gamma} \chi\left(||\gamma||_{\sigma} \right)\cdot \exp{\left(-s\cdot ||\gamma||_{\sigma}\right)}\cdot D_{\gamma}$$ where $D_{\gamma}$ is the Dirac measure at $\gamma$. Taking the weak-* limit of $\mu_s$ as $s \searrow \delta$ gives us a Patterson--Sullivan measure $\mu$ of dimension $\delta$ with support in $\Lambda_{\Gamma}$ (see details in \cite[Section 4]{2024arXiv240409713B}).

   Since $\mathcal{U}_T(\Lambda_{K_2^c})$ is open for all $T$, the weak-* convergence gives \begin{equation}\label{19}
        \mu(\mathcal{U}_T(\Lambda_{K_2^c})) \leq \liminf_{s \searrow \delta} \mu_s(\mathcal{U}_T(\Lambda_{K_2^c})).
    \end{equation}
     By construction $\mathcal{U}_T(\Lambda_{K_2^c})\subset U_T(K_1,F)$ and $\mu_s$ is supported on $\Gamma$, hence it suffices to examine the behavior of $\mu_s(U_T(K_1,F)\cap \Gamma)$ to prove the lemma. 

For every $g \in \Gamma$, define $\mathcal{O}_{K_1}(g.K_1) \subset \Gamma$ as the set of points $y \in \Gamma$ such that there exist $ v\in K_1$ and $t_g>0$ where $$\tilde{\psi}^{t_g}(v) \in g.K_1, \text{ } \Tilde{\psi}^{[0,t_g]}(v)|\cap \Gamma.K_1\subset K_1 \cup g.K_1 \text{, and }  y.K_1 \cap \Tilde{\psi}^{[t_g,\infty)}(v) \neq \emptyset.$$
For all $\gamma \in \Gamma \cap U_T(K_1,F),$ there exist $v \in K_1$, $h \in F$ and $g \in \Gamma\smallsetminus F$ and $t_h, t_{g}>0$ such that  $$\tilde{\psi}^{[t_h,t_{g}]}(v) \cap \Gamma.K_1 \subset (h.K_1 \cup g.K_1) \text{ and } \tilde{\psi}^s(v) \cap \gamma.K_1 \neq \emptyset.$$ 
Let $R: = \max_{v' \in K_1} \{t >0 \text{ }|\text{ } \tilde{\psi}^t(v) \in F.K_1\}.$ Then by definition of $U_T(K_1,F)$, we have $$ t_{h^{-1}g}:= t_{g}-t_h > T-R \text{, } h^{-1}g \in \Gamma_{K_1} \text{, and } h^{-1}\gamma \in \mathcal{O}_{K_1}(h^{-1}g .K_1).$$
Then 
\begin{equation}\label{11}
\mu_s(U_T(K_1,F) \cap \Gamma )  \leq \mu_s\left(\bigcup_{h\in F} \bigcup_{\substack{h^{-1}g \in \Gamma_{K_1}  \\ t_{h^{-1}g}>T-R}} h.\mathcal{O}_{K_1}(h^{-1}g.K_1)  \right).
\end{equation} 
 We now estimate the measure of $h.\mathcal{O}_{K_1}(h^{-1}g.K_1)$. Let $\mathcal{O}_{g^{-1}.K_1}(K_1):= g^{-1}. \mathcal{O}_{K_1}(g.K_1)$ for all $g \in \Gamma$. Note that $$\mu_s\left(h.\mathcal{O}_{K_1}(h^{-1}g.K_1) \right)= \left(g^{-1}_*\mu_s\right)\left( \mathcal{O}_{g^{-1}h.K_1}(K_1)\right).$$ Let $\bar{g}:= h^{-1}g.$
The measures $\mu_s$ and $(g^{-1})_* \mu_s$ are related as follow: for every $g \in \Gamma$ and $A \subset \Gamma \sqcup \Lambda_{\Gamma}$ measurable, we have
\begin{equation}\label{17}
  \left(g^{-1}_*\mu_s\right)\left(A \right) \leq  \int_{\xi \in A} \chi_{g^{-1}}(\xi) \exp{\left(-s\cdot f_{g^{-1}}^-(\xi)\right)} d\mu_s(\xi)
\end{equation}
where $\chi_{g^{-1}}$ extends the function $\chi$ that appears in the modified Poincar\'e series from $\Gamma $ to identity on the boundary,
and $f_{g^{-1}}^-$ is a lower semicontinuous extension of the function $ ||g\xi||_{\sigma} -||\xi||_{\sigma}$ to $M$. Readers can check the proof of Theorem 4.1 in \cite{2024arXiv240409713B} for details.

Since $\mathcal{O}_{\bar{g}^{-1}K_1}(K_1) \subset \Gamma$, we can write the integral in Equation \ref{17} as
\begin{equation}\label{18}
      \mu_s(h.\mathcal{O}_{K_1}(\bar{g}.K_1))= \int_{\xi \in \mathcal{O}_{\bar{g}^{-1}K_1}(K_1) } \frac{\chi(||g\xi||_{\sigma})}{\chi(||\xi||_{\sigma})} \exp{\left(-s\cdot \left(||g\xi||_{\sigma} -||\xi||_{\sigma}\right)\right)} d\mu_s(\xi).   
\end{equation}

Since $\bar{g}= h^{-1}g$ and $h \in F$ finite, according to Corollary \ref{cor of geod ball and shadow} and the first and second property in Fact \ref{prop of cocycle}, we can find $C:= C(\eps)$ such that $$||\bar{g}||_{\sigma}-C\leq ||g\xi||_{\sigma}-||\xi||_{\sigma}\leq ||\bar{g}||_{\sigma}+C$$ for all $g$ such that $\bar{g} \in \Gamma_{K_1}$ for some $h \in F$ and $t_{\bar{g}} > R'$ and all $\xi \in \mathcal{O}_{\bar{g}^{-1}K_1}(K_1).$

Fix $\eps'>0$, let $F_{\eps'} \subset \Gamma $ be the finite set such that $\chi(||\xi||_{\sigma}+t) > \exp (\eps't)\cdot \chi(||\xi||_{\sigma})$ for all $\xi \in F_{\eps'}$ and $t >0.$
We can rewrite Equation \ref{18} as 
  \begin{equation}
      \begin{split}
       \mu_s\left(h.\mathcal{O}_{K_1}(\bar{g}.K_1) \right)& \leq  C\cdot \int_{\xi \in \mathcal{O}_{\bar{g}^{-1}K_1}(K_1)} \frac{\chi(||g\xi||_{\sigma})}{\chi(||\xi||_{\sigma})} \exp{\left(-s\cdot ||\bar{g}||_{\sigma} \right)} d\mu_s(\xi)\\
       &\leq C\cdot \int_{\xi \in \mathcal{O}_{\bar{g}^{-1}K_1}(K_1) } \exp{\left(-(s-\eps')\cdot ||\bar{g}||_{\sigma} \right)} d\mu_s(\xi)  +\mu_s(F_{\eps'}) \\
       &\leq C\cdot \exp{\left(-(s-\eps')\cdot ||\bar{g}||_{\sigma}\right)} \cdot \mu_s(\mathcal{O}_{\bar{g}^{-1}K_1}(K_1) ) +\mu_s(F_{\eps'})\\
      & \leq C\cdot \exp{\left(-(s-\eps')\cdot ||\bar{g}||_{\sigma}\right)} +\mu_s(F_{\eps'}) .
  \end{split}
  \end{equation}  
  as $\mu_s$ are all probability measures.

 Since $F_{\eps'}$ is bounded, and the support of $\mu$ is in $\Lambda_{\Gamma}$, $\mu_s(F_{\eps'}) \rightarrow 0$, Equation \ref{19} turns into \begin{equation*}
     \begin{split}
     \mu(\mathcal{U}_T(\Lambda_{K^c_2}) ) &\leq \liminf_{s \searrow\delta} C\cdot  \sum_{h \in F} \sum_{\substack{\bar{g} \in \Gamma_{K_1} \\ t_{ \bar{g}}>T-R} } \exp{(-(s-\eps')\cdot ||\bar{g}||_{\sigma})}\\
     & \leq C \cdot |F|\cdot \sum _{\substack{\bar{g} \in \Gamma_{K_1} \\ t_{\bar{g}>T} }} \exp{(-(s-\eps')\cdot ||\bar{g}||_{\sigma})},
 \end{split}
 \end{equation*} where $|F|$ is the cardinality of the finite set $F$.

 The SPR condition ensures that $\delta> \delta_{\sigma}(\Gamma_{K_1})$, hence for all $\eps'$ small such that $\delta-\eps' > \delta_{\sigma}(\Gamma_{K_1})$, the series $$\sum_{\gamma \in \Gamma_{K_1} } \exp{(-(s-\eps')\cdot ||\gamma||_{\sigma})}$$ converges for all $s > \delta$. This implies as $T\rightarrow \infty$, the tail decreases exponentially fast, i.e. it converges to $0$, which concludes the proof of Proposition \ref{measure of U_T}.
\end{proof}

\begin{proof}[Proof of Theorem \ref{spr implies divergent}]

By Lemma \ref{nbhd of Lambda_K2c}, $\Lambda_{K_2^c} \subset \bigcap_{T> 0} \mathcal{U}_T(K_1,F)$, and we then have 
    \begin{equation}\label{10}
    \begin{split}
        \mu(\mathcal{L}_{K_2^c})& \leq \mu\left(\Gamma.\left(\bigcap_{T> 0} \mathcal{U}_T(K_1,F)\right)\right)\\
        &\leq \sum_{\gamma \in \Gamma} (\gamma_*\mu)\left(\bigcap_{T\geq 0} \mathcal{U}_T(K_1,F)\right).
    \end{split}   
    \end{equation}
    
    Proposition \ref{measure of U_T} ensures that $\mu\left(\bigcap_{T>0} \mathcal{U}_T(K_1,F)\right) =0$, and since all $\gamma_*\mu$ are absolute continuous with respect to each other, this implies $\gamma_*\mu(\bigcap_{T>0}\mathcal{U}_T(K_1,F)) = 0$. We then have $\mu(\mathcal{L}_{K^c_2})=0$, and respectively $\mu(\Lambda_{\Gamma}^{con})= 1$. Then the dichotomy in Theorem \ref{dichotomy} implies $\Gamma$ is divergent, which concludes the proof of Proposition \ref{spr implies divergent}.
\end{proof}

\section{Examples of SPR Subgroups}
As mentioned in the introduction, one of the motivations for this paper is to find examples of subgroups of higher rank Lie groups beyond relative Anosov groups that still admit finite BMS measures.

\subsection{Schottky Products}\label{new example} Now we present a way to construct such examples. This type of construction originated in Maskit \cite{5cd946a3-4546-32c2-8666-da16c4b98fac} for discrete subgroups of isometries of hyperbolic manifolds, and are generalized to many different settings (see Schapira--Tapie \cite{2018arXiv180204991S} in the case of complete negatively curved Riemannian manifolds, Dey--Kapovich \cite{2023arXiv230102354D} for Anosov groups, and Danciger--Gu\'eritaud--Kassel \cite{2024arXiv240709439D} for automorphisms of properly convex domains in projective spaces etc.). 

Let $N$, $H$ act on a compact metrizable space $M$ as convergence groups. The subgroups $N$ and $H$ are said to be in \emph{Schottky position} if there exist disjoint compact sets $U_N$, $U_H \subset M$ such that $$n. \left(M \setminus U_N\right) \subset U_N \text{ and } h. \left(M \setminus U_H\right) \subset U_H $$ for all $n \in N\setminus \{e\}$ and all $h \in H \setminus \{e\}$. 

By a ping-pong lemma type argument, the group  generated by $N$ and $H$, called the \emph{Schottky product} of $N$ and $H$, is a free product. Let $\Gamma:= N\ast H$ denote the Schottky product, and we further assume that this Schottky product acts as a convergence group. 




We want to show that the entropy at infinity behaves nicely under the Schottky product. For the rest of the section, let $(\sigma,\bar{\sigma}, G)$ be a continuous GPS system associated with $\Gamma$. Let $||\cdot||_{\sigma}$ be the $\sigma-$magnitude defined on $\Gamma$ associated with $\sigma$. Let $\sigma'$ be an expanding cocycle that defines the action of $\Gamma$ on $\Tilde{\Omega}_{\Gamma}= \Lambda_{\Gamma}^{(2)} \times \R$. Note that since $H, N$ are subgroups of $\Gamma$, $\Tilde{\Omega}_H$ and $\Tilde{\Omega}_N$ are subspaces of $\Tilde{\Omega}_{\Gamma}$, $(\sigma,\bar{\sigma}, G)$ is also a continuous GPS system for $N$ and $H$, and the action of $H$, $N$ on $\Tilde{\Omega}_H$, $\Tilde{\Omega}_N$ are also defined by $\sigma'$,  and hence is compatible with the $\Gamma$-action.

\begin{thm}\label{relation of entropy at infinity}
    Let $N,H,$ and $\Gamma$ be defined as above. We have $$\delta_{\infty}(\Gamma) = \max \{\delta_{\infty}(H), \delta_{\infty}(N)\},$$ where $\delta_{\infty}$ is the entropy at infinity defined in Section \ref{entropy at infinity}.
\end{thm}

The proof of this theorem uses the ideas in Section 8 of \cite{2016arXiv161003255S} and Section 7 of \cite{2018arXiv180204991S}. 
Since $\Gamma$ is a free product of $H$ and $N$, we can express every element $g \in \Gamma$ as a product of elements alternating between $N$ and $H$, i.e $$g = h_1n_1 \cdots h_k n_k$$ with $h_1 \in H, n_k \in N$, and $h_i \in H \setminus \{e\}$ for all $k\geq i >1$ and $n_j \in N\setminus \{e\}$ for all $1\leq j<k.$ Using this, we define $$\Gamma^H:=\{g \in \Gamma\text{ }|\text{ } h_1 \in H\setminus \{e\}\} \quad \text{ and }\quad \Gamma^N:=\{g \in \Gamma\text{ }|\text{ } h_1 \in N\setminus \{e\}\},$$ which gives a partition $\Gamma = \Gamma^H \sqcup \Gamma^N \sqcup \{e\}.$

\begin{lemma}\label{visibility}
    Fix disjoint compact sets $U_N,U_H \subset \Lambda_{\Gamma}$ that allows us to perform a Schottky product between $H$ and $N$, there exist:
    \begin{itemize}
    \item A compact set $K \subset \Tilde{\Omega}_{\Gamma}$ such that for all $x \in U_N$, $y \in U_H$, the flow line defined by $x$ and $y$ intersects $K$ non-trivially.
        \item Compact sets $U_H', U_N' \subset \Tilde{\Omega}_{\Gamma}\sqcup M$ such that $U_H' \cap U_N' = \emptyset$ and $n.K \subset  U_N'$ and $h.K \subset U_H'$ for all but finitely many $n \in \Gamma^N$ and $h \in \Gamma^H$;
        
    \end{itemize}
\end{lemma}

\begin{proof}

We prove the first point by contradiction. If not, we can find sequences $x_n \in U_N$ and $y_n \in U_H$, and an exhausting sequence of compact sets $K_n$ with $K_n \subset K_{n+1}$ such that $(x_n, y_n ,t) \notin K_n$ for all $t \in \R$. After taking subsequences, $x_n \rightarrow x \in U_N$ and $y_n \rightarrow U_H$. By construction, $(x,y,t) \notin \Tilde{\Omega}_{\Gamma}$ for all $t$, which would imply $x= y$. This gives a contradiction, as $U_H$ and $U_N$ are disjoint. 

   In order to prove the second point, we just need to realize $\Lambda_{\Gamma} \subset U_H \cup U_N.$ If not, there exists $x \in \Lambda_{\Gamma}\setminus (U_H \cup U_N).$ By definition of a limit point, there exists a sequence $\{g_n\} \subset \Gamma$, and $y \in M$ such that $g_n|_{M\setminus \{y\}}$ converges to $x$ locally uniformly. Up to taking subsequence, we can assume $\{g_n\} \subset \Gamma^H$ (or $\{g_n\} \subset \Gamma^N$). This gives a contradiction, as for any point $x' \in M \setminus (U_H \cup U_N),$ the Schottky position property ensures $g_n.x' \in U_H$ (or $g_n.x' \in U_N$), hence they cannot converge to $x$. This argument also shows that the limit point for any distinct sequence $\{g_n\} \subset \Gamma^H$ (or $\Gamma^N)$ is in $U_H$ (or $U_N$ respectively), which guaranties the existence of the compact sets $U_H'$ and $U_N'$ as required.
    \end{proof}

Let $K, U_H'$, $U_N'$ be compact sets that satisfy Lemma \ref{visibility}. 
Note that by repeating the proof of the second part in Lemma \ref{visibility}, we can enlarge $K$ so that $K$ contains $K_0$ defined in Section \ref{Gamma_K}, and every flow line that passes through $U_H'$ and $U_N'$ will also pass through the enlarged set which, to abuse notation, will also be denoted by $K$. 

\begin{lemma}\label{relating Gamma_K}
There exists a finite set $S \subset \Gamma$ such that $$\Gamma_{K}  \subset S\cup  H \cup N.$$    
\end{lemma}

\begin{proof}
    For $g \in \Gamma \setminus (H\cup N)$, we can rewrite $$g = nhg' \text{ or } g = hng'$$ for some $h \in H\setminus\{e\}$ and $n \in N\setminus\{e\}$ and $g' \in \Gamma.$ 

    Now, if $g = nhg'\in \Gamma_K$ (the other case of $g = hng'$ can be proven analogously), then $hg'n \in \Gamma_{n^{-1}.K}$, i.e. there exists $v \in n^{-1}.K$ such that $\Tilde{\psi}^{[0,T]}(v)$ only intersects $n^{-1}.K$ and $hg'K$ for some $T >0$. By construction, except for finitely many $n $ and $h$, $n^{-1}.K \subset U_N' $ and $hg'.K \subset U_H'$. However, this raises a contradiction because $\Tilde{\psi}^{[0,T]}(v)$ has to intersect $K$. This shows that $\Gamma_K \cap (H\cup N)^c $ is finite, hence proves the lemma.
\end{proof}

\begin{proof}[Proof of Theorem \ref{relation of entropy at infinity}]
In order to prove the theorem, we just need to show that for all compact sets $K_1, K_2$ with $K \subset K_1\subset K_2$, there exists a finite set $F$ such that \begin{equation}\label{13}
        \Gamma_{K_2} \subset \left(S \cup H_{K_2} \cup N_{K_2}\right) \subset S \cup \left(\bigcup_{g ,g' \in 
    F}g\Gamma_{K_1} g'\right)
    \end{equation} where $H_{K_2}, N_{K_2}$ are the corresponding subsets outside of $K_2$ in $H$ and $N$ respectively. This containment would imply $$\delta_{K_2^c} (\Gamma) \leq \max \{\delta_{K_2^c}(H), \delta_{K^c_2}(N) \} \leq \delta_{K_1^c}(\Gamma).$$ Taking a limit as $K_1$ and $K_2$ get bigger gives us the desired equality. 

    Let us now prove the containment in Equation \ref{13}. The left inclusion is a consequence of the previous lemma. If $h \in \Gamma_{K_2} \cap H$, then there exists a flow line that passes through $K_2$ and $h.K_2$, and it does not intersect $\Gamma.K_2 \supset H.K_2$ in between, hence $h \in H_{K_2}. $ Similarly one can show $\Gamma_{K_2} \cap N \subset N_{K_2}.$ Then Lemma \ref{relating Gamma_K} implies $$\Gamma_{K_2}\setminus S \subset H_{K_2} \cup N_{K_2}.$$

    The right inclusion needs more work. We again just show there exists a finite set $F$ such that $H_{K_2} \subset \bigcup_{g, g' \in F}g\Gamma_{K_1} g'$ and the proof for $N_{K_2}$ goes similarly. 

    If $h \in H_{K_2}$, then there exists $v_h \in K_2$ such that $\tilde{\psi}^{[0,\infty)}(v_h)$ intersects $H.K_2$ only at $K_2$ and $h.K_2$ on some initial segment $[0,T]$. Let $g \in \Gamma\setminus \{e\}$ be such that $g.K_1  \cap \Tilde{\psi}^{[0,T]}(v_h) \neq \emptyset$. 

    \begin{claim}
         The intersection $g.K_1  \cap \Tilde{\psi}^{[0,T]}(v_h) $ is contained in $\Tilde{\psi}^{[0,r]}(v_h)$ or $\Tilde{\psi}^{[T-r,T]}(v_h)$ for some $r >0$ independent of $h$.
    \end{claim}
    \begin{claimproof}
    If not, then there exist sequences $\{g_n\} \subset \Gamma$, $\{h_n\} \subset H_{K_2},$ $v_n \in K_2$, and $t_n \rightarrow \infty$ such that $$\Tilde{\psi}^{[0,t_n]}(v_h) \cap H.K_2 \subset K_2 \cup h_n.K_2$$ and $$g_n.K_1 \cap \Tilde{\psi}^{[n,t_n-n]}(v_n) \neq \emptyset.$$ Up to taking subsequences, we can assume $g_n$ start with non-trivial elements all in $H $ or all in $N$. 

\emph{Case 1:}
    Let $\{g_n\} \subset \Gamma^N$. Then for all but finitely many $n$, $g_n.K_1 \cap U_N' \neq \emptyset.$ More importantly, up to taking subsequences, $g_n.K_1 \rightarrow x \in U_N$, while $v_n \rightarrow v \in K_2, $ and $h_n.K_2 \rightarrow y \in U_H$. This cannot happen as $U_H$ and $U_N$ are disjoint, but by construction $x= y = \lim_{s \rightarrow \infty} \Tilde{\psi}^s(v).$

\emph{Case 2:}
    Now we prove the inclusion when $\{g_n\} \subset \Gamma^H$. Let $h_n'$ be the non-trivial element in $H $ that starts each $g_n$. We can re-center everything by $(h_n')^{-1}$. 

If $\{h_n\}$ are bounded, then by taking further subsequences we can assume $h_n= h' \in H$ for all $n$. Now $(h')^{-1}g_n$ will again send $K_1$ to some points $x \in U_N$, and $(h')^{-1}h_n$ send $K_2$ to some $y \in U_H$, and by the same argument in case 1, we have a contradiction.

If $\{h_n\}$ are unbounded, then after re-centering we get $h_n^{-1}g_n.K_2 \subset U_N' $ and $h_n^{-1}.K_2 \subset U_H'$ for large $n$. However, Lemma \ref{visibility} then implies that the flow line $h_n^{-1}\tilde{\psi}^{[0,\infty)}(v_n)$ would pass through $K_2$, this contradicts the fact that $h_n $ are all in $H_{K_2}.$
\end{claimproof}
    With the claim above, we can form the finite set $$F:= \bigg\{g \in \Gamma\text{ }|\text{ }g.K_2 \cap \left(\bigcup_{s\in [o,r]}\Tilde{\psi}^s(K_2)\right) \neq \emptyset\bigg\}.$$
    Using the same argument as in Proposition \ref{containment of Gamma_K}, we have $$H_{K_2} \subset \bigcup_{g,g'\in F}g\Gamma_{K_1}g'$$ which concludes the proof of Theorem \ref{relation of entropy at infinity}. 
\end{proof}

\begin{example}\label{free product of anosov}
Let $G$ be a semi-simple Lie group, and $P_{\theta}$ be a parabolic subgroup in $G$. Let $N,H$ be two $P_{\theta}-$Anosov subgroups of $G$, then they act as convergence subgroups on the space $\mathcal{F_{\theta}}: = G/P_{\theta}$. If $H$ and $N$ are in Schottky position, then Theorem A in \cite{2023arXiv230102354D} implies their free product $H \ast N$ is $P_{\theta}$-Anosov, hence also acts as a convergence group on $\mathcal{F}_{\theta}$. Let $\Lambda_H, \Lambda_N,$ and $\Lambda_{N\ast H}$ denote the corresponding limit sets in $\mathcal{F}_{\theta}$.

Pick a linear functional $\phi$ that is positive on the Benoist limit cone $\mathcal{B}_{\theta}(H \ast N)$. Composing $\phi$ with the Cartan projection from $G$ to the positive Weyl chamber of its Cartan subalgebra $\mathfrak{a}$ gives an expanding cocycle. We denote this cocycle also by $\sigma_{\phi}$. Pre-composing $\phi$ with a Cartan involution on $\mathfrak{a}$ gives another expanding cocycle $\bar{\sigma}_{\phi}$, and there exists a Gromov product $G$ on $\mathcal{F}_{\theta}^{(2)}$ such that the triple $(\sigma_{\phi}, \bar{\sigma}_{\phi}, G)$ forms a GPS system. 

Assume further that $H$ is the image of some convex cocompact subgroup in $SO(d,1)$ embedded into $G$ and it has a normal subgroup $H'$ such that $H/H' $ is amenable. It is known (see e.g. \cite{roblin:hal-00018990}, \cite{Brooks1985}, \cite{2017arXiv170206115D}) that the critical exponent of $H'$ equals to $H$ in this setting. 
Moreover, we can pick $N$ such that $\delta_{\phi}(N) > d \geq \delta_{\phi}(H)= \delta_{\phi}(H')$. Then Theorem 4.1 in \cite{2023arXiv230411515C} implies that $$\delta_{\phi}(H'\ast N) > \delta_{\phi}(N)=\max\{\delta_{\phi}(H'), \delta_{\phi}(N)\} \geq \delta_{\infty}(H') = \delta_{\infty}(H'\ast N).$$ This implies $H'\ast N$ is SPR, and thus admits a finite BMS measure, but it need not be relative Anosov.

\end{example}
\begin{rmk}
This is a direct analog of the example of geometrically infinite subgroups of $Isom(\H^d)$ for $d \geq 4$ that admits a finite BMS measure constructed by Peign\'e \cite{peigne}.
\end{rmk}

The reason to start with two Anosov subgroups in the previous is to ensure that their free product is indeed a convergence group acting on $\mathcal{F}_{\theta}$, so that the assumptions of Theorem \ref{relation of entropy at infinity} are satisfied. This assumption is not necessary, as we will see in the next example. Although the free product of convergence groups may not be a convergence group, when both groups are transverse subgroups of a higher rank Lie group, their free product admits a representation where the image is transverse, and hence a convergence group.

Often times, we can embed a $P_{\theta}$-transverse subgroup into $PGL(\R^d)$ as a $P_{1,d-1}$-transverse linear group \cite[Theorem 6.2]{2023arXiv230411515C}. Hence we will just state the next example in the linear group case.

\begin{example}\label{projective geom}
Let $G:= PGL(d,\R)$ and let $\theta = \{1,d-1\}$. Then $$\mathcal{F}_{\theta}:= \{[x,V]\text{ }|\text{ }  x \in \P(\R^d) \text{ and } V \text{ a projective hyperplane containing } x\}.$$ Take two $P_{1,d-1}-$transverse subgroups $H_1,H_2 \subset G$ that preserve two open properly convex domains $\Omega_{\sigma'},\Omega_2 \subset \P(\R^d)$. 

Assume the corresponding limit sets $\Lambda_{i}:=\Lambda_{\Omega_i}(H_i)$ does not fill the whole boundary $\del \Omega_i$ for $i \in \{1,2\}$. Then a result by Danciger--Gu\'eritaud--Kassel \cite[Theorem 1.6]{2024arXiv240709439D} applies: there exists $g \in PGL(d,\R)$ such that the representation $\rho:H_1 \ast H_2 \rightarrow H_1 \ast gH_2g^{-1} $ is discrete and faithful, and its image preserves an open properly convex domain $\Omega \subset \P(\R^d).$

This properly convex domain $\Omega$ is constructed as taking the convex hull of a tree of properly convex domains where the vertices consist of translates of $\Omega_{\sigma'}' \subset \Omega_{\sigma'}$ and $g.\Omega_2' \subset g.\Omega_2$ under the action of $\rho(H_1 \ast H_2)$. It satisfies the condition that every three adjacent domains $(\Omega_{n-1},\Omega_n, \Omega_{n+1})$ in the tree are in what Danciger--Gu\'eritaud--Kassel called occultation position. The occulation position requires that $\Omega_i \cap \Omega_{i+1} \neq \emptyset$ for $i \in \{n-1,n\}$, and every projective line that passes $\overline{\Omega_{n-1}}$ and $\overline{\Omega_{n+1}}$ has to intersect $\Omega_n$; in other words, the domain $\Omega_n$ in the middle blocks $\overline{\Omega_{n-1}}$ and $\overline{\Omega_{n+1}}$ from ``seeing" each other. We refer the reader to Section 2 and 4 in \cite{2024arXiv240709439D} for a more detailed description of how the sets $\Omega_i'$ and $\Omega$ are constructed.

The construction implicitly forces $\Gamma:= \rho(H_1 \ast H_2)$ to be $P_{1,d-1}$-transverse. To see this, we only need to check that the limit set $\Lambda_{\Omega}(\Gamma)$ is transverse. Notice that any limit point of $\Gamma$ either lies in some translate of $\Lambda_i$ or is the limit of a sequence of convex domains in the tree. Let $x,y$ be two distinct limit points of $\Gamma$. If they lie in $\rho(h).\Lambda_i$ for some $h \in H_1\ast H_2$ and $i \in \{1,2\}$, then transversality of $H_i$ implies the transversality of $x$ and $y.$ If they lie in different translates of $\Lambda_i$, or if one of them is a limit of convex domains and the other lies in $\rho(h).\Lambda_i$, then the occultation position ensures that the projective line passes through $x,y$ intersects $\Omega$. If both are limits of convex domains in the tree, let $h \in H_1 \ast H_2$ be the element such that the geodesic rays from $e$ to $x$ and $y$ part, then the occultation position again forces the projective line passing through $x$ and $y$ to also pass through $\rho(h).\Omega_{\sigma'}'.$

If we further assume that $H_1,H_2$ are SPR with respect to a GPS system defined by some linear functional $\phi$, then Theorem \ref{relation of entropy at infinity} implies $$\delta_{\phi}(\Gamma)\geq \max\{\delta_{\phi}(H_1), \delta_{\phi}(H_2)\} > \max\{\delta_{\infty}(H_1), \delta_{\infty}(H_2)\} = \delta_{\infty}(\Gamma),$$ which in turn shows $\Gamma$ is also SPR.
\end{example}

\subsection{Relative Anosov Groups are SPR}\label{rel anosob}
We would also like to show that relative Anosov subgroups fit into the framework of SPR subgroups as they are an important class of groups that admit finite BMS measures.
In order to make the definition of relative Anosov subgroups in a higher rank Lie group precise, we first introduce a class of convergence groups that admits a geometrically finite action, and show that they are strongly positively recurrent.

\begin{defn}[Geometrically Finiteness]
Let $M$ be a compact perfect metrizable space. A convergence subgroup $\Gamma \subset Homeo(M)$ is \emph{geometrically finite} if every $\eta \in \Lambda_{\Gamma}$ is either 
\begin{itemize}
    \item \emph{conical} in the sense of convergence group action, as defined in Section \ref{conv group action},
    \end{itemize}
or
\begin{itemize}
    \item \emph{bounded parabolic}, that is, every $\gamma \in Stab_{\Gamma}(\eta)$ is parabolic, and $Stab_{\Gamma}(\eta)$ acts on $M \setminus \{\eta\}$ cocompactly.
\end{itemize}
    
\end{defn}

Let $(\sigma, \overline{\sigma},G)$ again be a continuous GPS system for a geometrically finite convergence group $\Gamma \subset Homeo(M)$ with $\delta_{\sigma}(\Gamma) < +\infty,$ and let $\sigma'$ be the exapnding cocycle that defines the $\Gamma$ action on $\tilde{\Omega}_{\Gamma}$. Further, let $$\mathcal{P}: =\{\gamma \cdot Stab_{\Gamma}(\eta)\cdot\gamma^{-1} \text{ }| \text{ } \gamma \in \Gamma \text{ and } \eta \in \Lambda_{\Gamma} \text{ is bounded parabolic}\}.$$ For any $P \in \mathcal{P}$, let $\delta_{\sigma}(P)$ be the critical exponent of the Poincar\'e series 
$$Q_{P,\sigma}(s): = \sum_{g \in P} \exp\left(-s \cdot ||g||_{\sigma}\right).$$

\begin{thm}\label{geom finite is spr}
    If $\delta_{\sigma}(P) < \delta_{\sigma}$ for all $P \in \mathcal{P}$, then $\Gamma$ is strongly positively recurrent.
\end{thm}

This statement is a direct analog of Proposition 7.16 in Schapira--Tapie \cite{2018arXiv180204991S}. Combining this with Theorem \ref{spr implies finite BMS} gives us a criterion for finiteness of BMS measures associated to a geometrically finite convergence group.

\begin{cor}\label{geom fin and fin bms}
    Let $\sigma'$ be an expanding cocycle that defines a $\Gamma$ action on $\Tilde{\Omega}_{\Gamma}$, and let $\Omega_{\Gamma}$ be the associated quotient of the flow space $\Tilde{\Omega}_{\Gamma}.$ If for all $\eps > 0$, there exists $C$ such that \begin{equation}\label{relate magnitudes'}
        ||\gamma||_{\sigma'} \leq C\cdot \exp{\left(\eps \cdot ||\gamma||_{\sigma}\right)}
    \end{equation}
    for all $\gamma \in \Gamma$, then $m_{\sigma}(\Omega_{\sigma'}) < +\infty$ for the BMS measure $m_{\sigma}$ defined in Section \ref{ps and bms measure}.
\end{cor}
 
We will see in the next section that relative Anosov groups fall under the scope of strongly positively recurrent convergence groups, thus Theorem \ref{geom finite is spr} and Corollary \ref{geom fin and fin bms} recover many existing results on the finiteness of BMS measures of relative Anosov groups.

\subsubsection{\emph{Relative Anosov Groups}}

Let $\Gamma$ be a finitely generated group, $\mathcal{P}$ be a collection of finitely generated infinite subgroups of $\Gamma$. The pair $(\Gamma,\mathcal{P})$ is \emph{relative hyperbolic} if $\Gamma$ acts on a compact perfect metrizable space $M$ as a geometrically finite convergence subgroup and the maximal parabolic subgroups of $\Gamma$ are exactly conjugates of $P \in \mathcal{P}$.

Fixing a generating set $S$ of $\Gamma$, we can construct the Groves-Manning cusp space by gluing combinatorial horoballs to $Cay(\Gamma,S)$ (see Section 3 of \cite{2006math......1311G} for the explicit construction). The Groves--Manning cusp space is Gromov hyperbolic, and its Gromov boundary, denoted by $\p(\Gamma, \mathcal{P})$, serves as a compact perfect metrizable space where $\Gamma$ act on as a geometrically finite convergence group. 

There are various characterizations of relative Anosov subgroups (see Section 4 of \cite{2022arXiv220714737Z}), we will present just one here. Let $G$ be a non-compact semi-simple Lie group of higher rank with trivial center, let $P_{\theta} \subset G$ be a symmetric parabolic subgroup (here symmetric means, when conjugated by a longest word in the Weyl group of $G$, $P_{\theta}$ can be identified with the standard opposite parabolic subgroup $P_{\theta}^-$). Let $\mathcal{F}_{\theta}: =G/P_{\theta}$ be the associated partial flag manifold, and let $\mathfrak{a}_{\theta}$ be a subspace of the Cartan subalgebra associated to $\theta$. Given any $P_{\theta}-$divergent subgroup $\Gamma$ of $G$, one can define a limit set associated with $\Gamma$, denoted by $\Lambda _{\theta}(\Gamma)$, as $$\Lambda _{\theta}(\Gamma):=\{\lim_{n \rightarrow \infty} m_n P_{\theta} \text{ }|\text{ } \{g_n=m_n \cdot \exp(\kappa_{\theta}(g_n)) \cdot l_n\} \subset \Gamma \text{ escaping}\}$$ where $\kappa_{\theta}$ is the Cartan projection of $G$ to $\mathfrak{a_{\theta}}$, and the convergence is with respect to the topology on the partial flag manifold. We say $\Gamma$ is \emph{$P_{\theta}$-transverse} if $\Lambda_{\theta}(\Gamma)$ is a transverse set in $\mathcal{F}_{\theta}$. We refer the reader to a more detailed introduction of properties of semi-simple Lie groups in Section 2 of \cite{2023arXiv230411515C}.

\begin{defn}
A subgroup $\Gamma \subset G$ is \emph{$P_{\theta}$-Anosov relative to $\mathcal{P}$} if it is $P_{\theta}$-transverse, $(\Gamma, \mathcal{P})$ is a relatively hyperbolic pair, and there exists a boundary map $$\xi: \p (\Gamma, \mathcal{P})\rightarrow \mathcal{F}_{\theta}$$ that is $\Gamma$-equivariant and a homeomorphism onto $\Lambda _{\theta}(\Gamma).$
\end{defn}

There is an Iwasawa cocycle, defined by Quint in \cite{quintpsmeasurehr}, $B_{\theta}:\Gamma\times \mathcal{F}_{\theta}\rightarrow \mathfrak{a}_{\theta}$ that satisfies a vector version of the cocycle identity. As introduced in Example \ref{free product of anosov}, picking a linear functional $\phi \in \mathfrak{a}_{\theta}^*$ gives rise to a continuous GPS system $(\sigma_{\phi}, \bar{\sigma}_{\phi},G)$. Furthermore, since $\xi$ is $\Gamma-$equivariant and a homeomorphism between $\p(\Gamma, \mathcal{P})$ and $\Lambda_{\theta}(\Gamma)$, the geometrically finite convergence group action of $\Gamma $ on $\p(\Gamma, \mathcal{P})$ can be pushed to $\Lambda_{\theta}(\Gamma)$.  Moreover, for the relative Anosov pair $(\Gamma,\mathcal{P})$, Corollary 7.2 in \cite{2023arXiv230804023C} states that $\delta_{\sigma} (H) < \delta_{\sigma}(\Gamma)$ for all $H \in \mathcal{P}$. The following corollary is then a direct consequence of Theorem \ref{geom finite is spr}.

\begin{cor}\label{rel anosov is spr}
    Let $\Gamma \subset G$ be a $P_{\theta}-$Anosov group relative to $\mathcal{P}$, and let $\phi \in \mathfrak{a}_{\theta}^*$ be a linear functional with $\delta_{\sigma_{\phi}}(\Gamma)<+\infty$. Then $\Gamma$ is strongly positively recurrent, i.e. the entropy at infinity $\delta_{\infty}(\Gamma)$ defined with respect to the $\sigma-$magnitude $||\cdot ||_{\sigma}$ is strictly smaller than $\delta_{\sigma_{\phi}}(\Gamma)$.

\end{cor}

This corollary, combined with Corollary \ref{geom fin and fin bms}, encompasses many existing results that relate geometrically finiteness with the finiteness of BMS measure. It was first proven by Dal'bo--Otal--Peign\'e \cite{dalbootalpeigne} that geometrically finite negatively curved manifolds has finite BMS measure. Later, Blayac--Canary--Zhang--Zimmer\cite[Theorem 8.1]{2024arXiv240409718B} extended it to geometriaclly finite convergence groups under the assumption that there is an entropy gap between the critical exponent of each parabolic subgroups and that of $\Gamma$.

It is worth noting that the notion of SPR convergence groups is particularly nice for studying BMS measures associated to discrete subgroups of higher rank Lie group. One complication in higher rank symmetric spaces is that the unit tangent bundle of the space along with the geodesic flow on it is not the best flow space to work with, simply because there are too many directions. People have constructed various flow spaces associated to the same relative Anosov group and proved finiteness of the respective BMS measures.

In \cite{KimOh+2025+91+142}, Kim--Oh construct a flow space $\tilde{\Omega}_{\Gamma}:= \Lambda_{\theta}(\Gamma)^{(2)}\times \mathfrak{a}$, and project it to a one dimensional flow space $\tilde{\Omega}_{\Gamma, \phi}:=\Lambda_{\theta}(\Gamma)^{(2)}\times\R$ by sending $(\xi, \eta, v)$ to $(\xi, \eta, \phi(v))$, where $\phi \in \mathfrak{a}^*$ is a linear functional. They prove that the associated BMS measure on $\Gamma \setminus \tilde{\Omega}_{\Gamma, \phi}$ is finite \cite[Theorem 9.1]{KimOh+2025+91+142}. 
Another flow space associated to a relative Anosov group is constructed by Canary--Zhang--Zimmer in \cite{2023arXiv230411515C}. This involves taking a transverse representation $\rho$ of $\Gamma$ into $Aut(\Omega)\subset PGL(D, \R)$ for some properly convex domain $\Omega \subset \P(\R^D)$ as constructed in Section 4 in \cite{2023arXiv230411515C}. Define the subset $U:= \Lambda_{\Omega}(\rho(\Gamma))^{(2)}\times \R $ of the unit tangent bundle $ T^1\Omega$, where $ \Lambda_{\Omega}(\rho(\Gamma))$ is the full orbital limit set of $\rho(\Gamma)$, one can then study the geodesic flow on $U$. In a previous paper of the author \cite{2025arXiv250412448W}, Proposition 7.4 states that $\rho(\Gamma)\backslash T^1\Omega$ has finite BMS measure. 

Corollary \ref{geom fin and fin bms} provides a unifying framework for studying these flow spaces. Note that $\tilde{\Omega}_{\Gamma, \phi}$ and $U \subset T^1\Omega$ are the same topological space as $\Lambda_{\Omega}(\rho(\Gamma))$ is homeomorphic to $\Lambda_{\theta}(\Gamma)$ \cite[Observation 6.1]{2023arXiv230411515C}, the only difference is the cocycle that defines the $\Gamma$ actions on them. Kim--Oh use the same linear functional $\phi$ that defines $\tilde{\Omega}_{\Gamma, \phi}$, while Canary--Zhang--Zimmer use the Hilbert metric on the properly convex domain $\Omega$, which is quasi-isometric to $\log \frac{\mu_1}{\mu_D}(\rho(\cdot))$. Both choices satisfy the assumptions of Corollary \ref{geom fin and fin bms} (condition \ref{relate magnitudes'}), hence both finiteness theorems of the BMS measures can be derived from the corollary.


\subsubsection{\emph{Proof of Theorem \ref{geom finite is spr}}}
Let us now prove the theorem.
    We need to find a compact set $W \subset \Tilde{\Omega}_{\Gamma}$ such that $\delta_{W^c} < \delta_{\sigma}.$ We employ the ``thick-thin" decomposition of the flow space $\Omega_{\Gamma}$ proved by  Blayac--Canary--Zhu--Zimmer in \cite{2024arXiv240409718B} to construct such $W$. In order to make sense of this decomposition, we need to define horoballs in $\tilde{\Omega}_{\Gamma}$.

  As mentioned in Section \ref{Gamma_K}, the space $(M^{(2)} \times \R) \sqcup M =: \overline{M^{(2)} \times \R}$ admits a natural compactifying topology, and $\Gamma$ acts on $ \overline{M^{(2)} \times \R}$ also as a geometrically finite convergence group \cite[Section 7.1]{2024arXiv240409718B}.

\begin{defn}
    Let $x \in \Lambda_{\Gamma}$ be a bounded parabolic point. Let $K \subset \Lambda_{\Gamma}\setminus \{x\}$ be a compact set such that $Stab_{\Gamma}(x).K =\Lambda_{\Gamma}\setminus \{x\}$. Then for any compact neighborhood $N$ of $K$ in $\overline{M^{(2)} \times \R} $, we define the \emph{horoball associated to $N$ based at $x$} to be the subset $$H(x, N):= \Tilde{\Omega}_{\Gamma} \setminus (Stab_{\Gamma}(x).N).$$
\end{defn}

By a theorem of Tukia \cite{Tukia1998ConicalLP}, there are finitely many $\Gamma$ orbits of bounded parabolic fixed points. Let $\{x_1, ..., x_n\}$ be a collection that consists one representative from each orbit. For each $x_i$, fix a fundamental domain $K_i$ and the associated compact neighborhood $N_i$ as above. Then let $$\mathcal{H}:= \bigcup_{g \in \Gamma}\bigcup_{x_i} g.H(x_i, N_i).$$
Note that by construction, each $H(x_i,N_i)$ is open, hence $\Tilde{\Omega}_{\Gamma} \setminus  \mathcal{H}$ is closed.

\begin{thm}\cite[Theorem 7.5]{2024arXiv240409718B}\label{cocompact action}
    $\Gamma$ acts cocompactly on $\Tilde{\Omega}_{\Gamma} \setminus \mathcal{H}$.
\end{thm}

Moreover, Proposition 7.6 in \cite{2024arXiv240409718B} asserts that $N_i$ can be chosen in a way such that for all $g, h \in \Gamma$, and $1\leq i,j \leq n$, if $gx_i \neq hx_j$, then \begin{equation}\label{horoball disjoint}
    g.H(x_i, N_i) \cap h. H(x_j, N_j) = \emptyset.
\end{equation}

We can fix the choice of $N_i$ such that the property above holds. Theorem \ref{cocompact action} assures that, up to enlarging it, there exists a compact set $W$, such that $\Gamma.W = \Tilde{\Omega}_{\Gamma} \setminus \mathcal{H}$. There also exists $\gamma_i \in \Gamma$ for each $i$ such that the boundary of $\gamma_i. H(x_i,N_i)= H(\gamma_i. x_i, \gamma_i.N_i)$ intersects $W$ nontrivially. 

\begin{prop}\label{Gamma_W with P}
    There exists a finite set $S\subset \Gamma$ such that $$\Gamma_W =S \cup \left(\bigcup_{x_i} \bigcup_{g,h \in S} g\cdot  Stab_{\Gamma}(x_i)\cdot h\right).$$
\end{prop}
\begin{proof}
We first prove the inclusion $\Gamma_W\subset S \cup \left(\bigcup_{x_i} \bigcup_{g,h \in S} g\cdot  Stab_{\Gamma}(x_i)\cdot h\right)$.
    For any $\gamma \in \Gamma_W$, there exist $v \in W$ and $t_{\gamma} \in (0, \infty)$ such that $$\Tilde{\psi}^{[0, t_{\gamma}]}(v) \cap \Gamma.W \subset W \cup \gamma.W.$$ There exists $R >0$ such that for $\gamma \in \Gamma_W$ with $t_{\gamma}>2R$, we have $$\Tilde{\psi}^{[R, t_{\gamma}-R]}(v) \cap \Gamma.W = \emptyset.$$ By construction, $\Gamma.W = \Tilde{\Omega}_{\Gamma} \setminus \mathcal{H}, $ then $\Tilde{\psi}^{[R, t_{\gamma}-R]}(v)$ stays in some horoball $\gamma_i.H(x_i,N_i)$. Therefore, by the same argument as in the proof of Lemma \ref{containment of Gamma_K}, there exists a finite set $S'$ such that for some $g,h \in S'$ and $v' \in g .W$, the flow line defined by $v'$ intersects $g.W$ and $\gamma h.W$ and no other translates of $W$ in between, and this initial segment stays in the horoball $\gamma_i.H(x_i,N_i)$. This implies $a = g^{-1}\gamma h \in Stab_{\Gamma}(g^{-1}\gamma_i.x_i)= (g^{-1}\gamma_i)Stab_{\Gamma}(x_i)(g^{-1}\gamma_i)^{-1} $. Taking $S= S' \cup \{\gamma_i\}\cup \{\gamma \in \Gamma_W \text{ with } t_{\gamma} \leq 2R\}$ proves the first inclusion.

   We prove the other inclusion by contradiction. Assume there exist $x\in \{x_1,\cdots,x_n\}$  and an escaping sequence $\{g_i\} \subset Stab_{\Gamma}(x) \smallsetminus \Gamma_{W}$. Then by definition of $\Gamma_W$, corresponding to each $g_i$, there exist $\gamma_i \in \Gamma$ and $t_i>0$ such that for all $v= (v^+,v^-,s_v) \in W$ and $ t_v >0$ with $\tilde{\psi}^{t_v}(v) \in g_i.W$, one has $\tilde{\psi}^{t_i}(v)  \cap \gamma_i.W \neq \emptyset.$  Fix such $v$, We now examine this by cases.
   
   \emph{Case 1}: 
   Assume $t_i $ and $t_v-t_i$ both converge to infinity as $i \rightarrow \infty$. By the choice of $W$, once the flow line $\tilde{\psi}^{[0,\infty)}(v)$ left $W$, it has to enter some horoball in $\mathcal{H}$ before it enters $\gamma_i.W$. We can pull everything back by $\gamma_i^{-1}$, and since by construction the set of horoballs are pairwise disjoint, we can assume that, up to taking subsequence, $\gamma_i^{-1}.H(x,N)$ stabilizes at some horoball $H(x',N') \in \mathcal{H}$ for all $i$, or equivalently $\gamma_i^{-1} x= x'$ for all $i$. We then have $\gamma_i^{-1} \in Stab_{\Gamma}(x')$ and $\gamma_i^{-1}g_i\gamma_i \in Stab_{\Gamma}(x')$. Moreover, since $t_i $ and $t_v-t_i$, $\gamma_i^{-1},$ and $\gamma_i^{-1}g_i\gamma_i$ both converge to $x'$ as $i$ goes to infinity. This gives a contradiction as the flow lines $\gamma_i^{-1}.\tilde{\psi}^{[0,\infty)}(v)$ pass through the compact sets $\gamma_i^{-1}.W, W$, and $\gamma_i^{-1}g_i.W$ for all $i$, implying $\gamma_i^{-1},$ and $\gamma_i^{-1}g_i\gamma_i$ must converge to distinct points on the boundary $M.$

   \emph{Case 2}: If $t_i$ or $t_v-t_i$ is bounded. We will prove the case when $t_i$ is bounded, and the other case can be proven similarly by re-centering everything by $g_i^{-1}$. Up to taking subsequences, we can assume $\gamma_i= \gamma$ for some $\gamma \in \Gamma$. We then pull everything back by $\gamma^{-1}$ and show $\gamma^{-1}g_i \gamma \in Stab_{\Gamma}(\gamma^{-1}x) \cap \Gamma_W$ for large $i$. If not, then we can find $h_i \in \Gamma$ such that every flow line that passes through $W$ and $(\gamma^{-1}g_i\gamma).(\gamma^{-1}W)= \gamma^{-1}g_i.W$ also intersects $h_i.W$. Repeating the process for the new triple $(W, h_i.W, \gamma^{-1}g_i \gamma.W)$ until we find some $\gamma' \in \Gamma$ such that $(\gamma')^{-1}g_i \gamma' \in Stab_{\Gamma}((\gamma')^{-1}x) \cap \Gamma_W$. Note that this process has to terminate within finite many steps. Otherwise, there exist infinitely many elements $\alpha_j \in \Gamma$ such that any flow line $\tilde{\psi}^{[0,\infty)}(v)$ that passes through $W$ and $g_i.W$ also passes through $\alpha_j.W$ for all $j$. Lemma \ref{geod ball and shadows} implies that $v^+ \in S_{\eps}(g_i) \cap S_{\eps}(\alpha_j)$ for $i, j$ large enough. Moreover, Fact \ref{properties of shadows} says that the shadows $S_{\eps}(g_i) $ converge to $x$, then closeness of shadows ensures $x \in S_{\eps}(\alpha_j)$ for large $j$. This implies that $x$ is conical, which cannot happen as $x$ is assumed to be bounded parabolic. We again have a contradiction and this concludes the proof of the proposition.
\end{proof}

Let $P_i \in \mathcal{P}$ be the stabilizer of $x_i$. Proposition \ref{Gamma_W with P} implies that $$\delta_{\sigma}(\Gamma_W)= \max \delta_{\sigma}(P_i).$$
This in turn implies $$\delta_{\sigma}(\Gamma_W) = \max \delta_{\sigma}(P_i) <\delta_{\sigma}(\Gamma).$$ We can also enlarge the sets $N_i$ (or equivalently shrink the horoballs) so that $W$ is large enough that it contains $K_0$ as defined in Section \ref{Gamma_K}. Hence, Remark \ref{rewrite delta infty}, combined with the inequality above, implies that $$\delta_{\infty} \leq \delta_{\sigma}(\Gamma_W) <\delta_{\sigma}(\Gamma).$$ This concludes the proof of Theorem \ref{geom finite is spr} as we find the desired compact set $W$.



It is also worth noting that the framework of strongly positively recurrent subgroups provides an alternative proof to the result that the Poincar\'e series of a relatively Anosov group $\Gamma$ diverges at its critical exponent $\delta_{\sigma}$ \cite[Theorem 8.1]{2023arXiv230804023C}. One can see this by realizing that with the compact set $W$ constructed in the proof above, the set $\mathcal{L}_{W^c}= \Gamma.\Lambda_{W^c}$ consists exactly of all bounded parabolic limit points. Then the proof of Theorem \ref{spr implies divergent} shows that $\mathcal{L}_{W^c}$ has measure zero, which in turn implies the conical limit set has full measure, and by the dichotomy (Theorem \ref{dichotomy}), the Poincar\'e series of $\Gamma$ diverges at its critical exponent.

\bibliographystyle{plain}
\bibliography{finiteBMS}
\end{document}